\documentclass[11pt]{amsart}
\usepackage{amssymb,amsmath,amsfonts,amscd,euscript}
\usepackage[dvipsnames]{xcolor}
\usepackage{tikz}
\usepackage{hyperref}
\usepackage{mathdots}
\usepackage{enumitem}
\usepackage[dvipsnames]{xcolor}

\newcommand{\nc}{\newcommand}

\numberwithin{equation}{section}
\newtheorem{thm}{Theorem}[section]
\newtheorem{prop}[thm]{Proposition}
\newtheorem{lem}[thm]{Lemma}
\newtheorem{cor}[thm]{Corollary}
\newtheorem{rem}[thm]{Remark}

\newtheorem{example}[thm]{Example}
\newtheorem{dfn}[thm]{Definition}

\newtheorem{conj}[thm]{{\bf Conjecture}}

\nc{\la}{\lambda}
\nc{\al}{\alpha }
\nc{\be}{\beta }
\nc{\ve}{\varepsilon }
\nc{\om}{\omega }

\nc{\hk}{\twoheadrightarrow}
\nc{\msl}{\mathfrak{sl}}

\newcommand{\bC}{{\mathbb C}}

\newcommand{\bZ}{{\mathbb Z}}

\newcommand{\bP}{{\mathbb P}}

\newcommand{\fh}{{\mathfrak h}}
\newcommand{\fn}{{\mathfrak n}}
\newcommand{\fg}{{\mathfrak g}}
\newcommand{\fb}{{\mathfrak b}}
\newcommand{\bd}{{\bf d}}
\newcommand{\bI}{{\bf I}}
\newcommand{\U}{{\mathrm U}}
\newcommand{\rk}{{\mathrm{rk}}}
\newcommand{\floor}[1]{\lfloor #1 \rfloor}
\newcommand{\ceil}[1]{\lceil #1 \rceil}

\begin{document}

\title[Dellac configurations and flag varieties]
{Symmetric Dellac configurations and symplectic/orthogonal flag varieties}

\author{Ange Bigeni}
\address{Ange Bigeni:\newline
Department of Mathematics,\newline
National Research University Higher School of Economics,\newline
Usacheva str. 6, 119048, Moscow, Russia}
\email{ange.bigeni@gmail.com}

\author{Evgeny Feigin}
\address{Evgeny Feigin:\newline
Department of Mathematics,\newline
National Research University Higher School of Economics,\newline
Usacheva str. 6, 119048, Moscow, Russia,\newline
{\it and }\newline
Skolkovo Institute of Science and Technology, Nobelya Ulitsa 3, Moscow 121205, Russia}
\email{evgfeig@gmail.com}

\begin{abstract}
The goal of this paper is to study the
link between the topology of the degenerate flag varieties and combinatorics of the Dellac configurations.
We define three new classes of algebraic varieties closely related to the degenerate flag varieties of types A and C.
The construction is based on the quiver Grassmannian realization of the degenerate flag varieties and odd symplectic and 
odd and even orthogonal groups. We study basic properties of our varieties; in particular, we construct cellular decomposition 
in all the three cases above (as well as in the case of even symplectic group). We also compute the Poincar\'e polynomials
in terms of certain statistics on the set of symmetric Dellac configurations.   
\end{abstract}

\maketitle

\section*{Introduction}
The goal of this paper is to develop further the link between the topology of the degenerate flag varieties 
and quiver Grassmannians from one side and combinatorics of Dellac configurations from the other side.

The PBW degeneration of classical flag varieties of type $A$ was introduced in \cite{F3} in Lie theoretic terms. The resulting 
varieties $F^a_N$ have very explicit linear algebra description. Namely, let $E$ be an $N$-dimensional vector space with 
a basis $e_1,\dots,e_N$ and let $pr_k:E\to E$ be the projection along the $k$-th basis vector to the linear span 
of the rest basis vectors. Then $F^a_N$ consists of collections $(V_1,\dots,V_{N-1})$ of subspaces of $E$ such that $\dim V_k=k$ and
$pr_{k+1}V_k\subset V_{k+1}$ for all $k$. Theses varieties are singular, but share many nice properties with their classical analogues, see 
\cite{F1,F2,FFi,H,LS}. It was shown that the varieties $F^a_N$ can be realized as Schubert varieties \cite{CL,CLL} and
as quiver Grassmannains \cite{CFR1}. The Betti numbers of the varieties  $F^a_N$ are given by the normalized median Genocchi 
numbers $1,2,7,38,295,\dots$ \cite{Ba,DR,DZ,B2,B3,B4,De,Du,F2}. 
These numbers have several definitions; the initial one due to Dellac (the most important for us)  is based on the notion of 
Dellac configurations -- $n\times 2n$ tableaux with $2n$ marked boxes satisfying certain conditions.
The Poincar\'e polynomials of the degenerate flag varieties thus
provide natural $q$-analogues of the normalized median Genocchi numbers (see \cite{F1,F2,HZ1,HZ2,ZZ}).     
These polynomials can be explicitly described in terms of certain inversion counting statistics on the set of 
Dellac configurations. 

The PBW degeneration procedure works for arbitrary simple Lie groups. The case of $Sp_{2n}$ was considered in
\cite{FFL}. It was shown that the corresponding degenerate flag varieties $SpF^a_{2n}$ can be explicitly described as follows.
Let $E$ be a $2n$-dimensional vector space with a basis $e_1,\dots,e_{2n}$ and a non-degenerate skew-symmetric form
defined by $(e_i,e_{2n+1-i})=1$ for $1\le i\le n$. Then  $SpF^a_{2n}$ is defined as a subvariety of $F^a_{2n}$ consisting
of collections $(V_1,\dots,V_{2n-1})$ such that $V_k=V_{2n-k}^\perp$. The combinatorics of the torus fixed points counting for
$SpF^a_{2n}$ is described in \cite{B1,FF}. In particular, it was shown in \cite{FF} that the number $r_n$ of torus fixed
points of 
$SpF^a_{2n}$ is given by the number of the so-called symplectic Dellac configurations (in what follows we call them
symmetric Dellac configurations, the reason will become clear shortly). These are $2n\times 4n$ Dellac 
configurations symmetric with respect to the center. The main theorem of 
\cite{B1} states that the numbers $r_n$ appears in \cite{RZ} and can be expressed in terms of the surjective pistols.
The sequence $r_n$ starts with $1,2,10,98,1594$.

There are several natural questions arising in this context.
\begin{itemize}
\item How to define the statistics on the set of symplectic Dellac configurations to compute the Poincar\'e
polynomials of $SpF^a_{2n}$?
\item How to compute the number $l_n$ of symmetric Dellac configurations of size $(2n+1)\times (4n+2)$?
\item Are the numbers $l_n$ related to the odd symplectic group?
\item Does there exist an orthogonal analogue of the symplectic story?
\end{itemize}

In this paper we answer all these questions except for the third one, which is addressed in the companion paper \cite{BF}.
In particular, we show in \cite{BF} that the numbers $l_n$ also pop up in \cite{RZ}. Here are the first several 
entries of this sequence: $1,3,21,267,5349$.
 
Let us give a general definition of the varieties we are interested in. Let $E$ be an $N$-dimensional complex vector space 
equipped with a symmetric or skew-symmetric bilinear form. We assume that the rank of the form is as large
as possible, i.e. the form has one-dimensional kernel in the skew-symmetric case for odd $N$ and has no kernel otherwise.
Let $n=\floor{N/2}$ and let us choose a basis $e_1,\dots,e_N$ of $E$ in such a way that $(e_k,e_{N+1-k})=1$  for $k=1,\dots, n$
and $(e_i,e_j)=0$ if $i+j\ne N+1$.  
We define the variety $F^a_E$ consisting of collections of subspaces $V_1,\dots,V_n$ of $E$
such that 
\begin{itemize}
\item $\dim V_k=k$, $k=1,\dots,n$,
\item $pr_{k+1}V_k\subset V_{k+1}$, $k=1,\dots,n-1$,
\item $(V_n,V_n)=0$ for even $N$ and $(V_n,pr_{n+1}V_n)=0$ for odd $N$.   
\end{itemize}
In particular, if $N$ is even and the form is symplectic, $F^a_E$ is isomorphic to the symplectic degenerate
flag variety. In the main body of the paper we denote the varieties $F^a_E$ by $SpF^a_N$ and $SOF^a_N$ depending
on the fixed form.

Here are the first basic properties of the varieties $F^a_E$.
\begin{thm}
For all $N \geq 1$, the variety $SpF^a_N$ is irreducible of dimension $\ceil{N/2}\floor{N/2}$.
The variety $SOF^a_N$ is reducible equi-dimensional of dimension $\ceil{\frac{N-1}{2}}\floor{\frac{N-1}{2}}$.
The number of irreducible components of $SOF^a_N$ is equal to $2^{\floor{N/2}}$.
\end{thm} 
In particular, $SOF^a_N$ are not isomorphic to the degenerate flag varieties for the group $SO_N$ (the latter are irreducible). 

The main theorem we prove in the paper is as follows.
\begin{thm}
The varieties $F^a_E$ admit cellular decomposition into complex affine cells. The Euler characteristics of 
$F^a_E$ coincides with the number of symmetric Dellac configurations. The Poincar\'e polynomials are computed
via two natural statistics responsible for symplectic and orthogonal cases.  
\end{thm} 

Our paper is organized as follows. 
In Section \ref{Classical} we recall basic facts about flag varieties for classical Lie groups.
In Section \ref{DFV} main results on the degenerate flag varieties in types $A$ and $C$ are collected.
In Section \ref{ESC} we construct and study cellular decomposition of the even symplectic degenerate flag varieties.
We construct a statistics on the set of symmetric Dellac configurations responsible for the Poincar\'e polynomials
of $SpF^a_N$.  
Section  \ref{OSC} is devoted to the study of the varieties $SpF^a_N$ for odd $N$. We prove that these varieties are irreducible,
construct cellular decomposition, compute Euler characteristics and Poincar\'e polynomials in terms of the symmetric Dellac configurations.
In Section  \ref{OC} we define and study the varieties $F^a_E$ for non-degenerate orthogonal form. 
We find the number of irreducible components, construct cellular decomposition and compute the Poincar\'e polynomials
in terms of yet another statistics on the set of symmetric Dellac configurations.
Finally, in Appendix examples of Poincar\'e polynomials of $F^a_E$ are given. Based on these examples,
we conjecture that the Poincar\'e polynomials in all types are unimodular.

\section{Classical picture}\label{Classical}
Let $\fg$ be a finite-dimensional simple Lie algebra with the Cartan decomposition
$\fg=\fn\oplus\fh\oplus \fn_-$ and let $\fb=\fn\oplus\fh$ be the Borel sualgebra. We denote by $G,N,H,N_-$ and $B$ 
the corresponding Lie groups. For a parabolic subgroup $P\supset B$ the corresponding flag variety is defined as the 
quotient $G/P$. In particular, if $P=B$ then $G/B$ is called the full flag variety. 

Let $\om_i$ and $\al_i$, $i=1,\dots,{\rm rk}\fg$ be the fundamental weights and simple roots.
Let $\la\in\fh^*$ be a dominant integral weight (i.e. $\la=\sum_{i=1}^{{\rm rk}\fg} m_i\om_i$, $m_i\in\bZ_{\ge 0}$) and let 
$V_\la$ be the corresponding finite-dimensional irreducible
highest weight $\fg$-module. In particular, for a highest weight vector $v_\la\in V_\la$ on has $\fn v_\la=0$ 
and $V_\la=\U(\fn_-)v_\la$. We consider a map $G/P\to \bP(V_\la)$ defined by mapping the class of an element $g\in G$
to the line spanned by $gv_\la$. The map is obviously $G$-equivariant and is an embedding if $\la$ is regular,
i.e. $m_i>0$ for ll $i$. 

Let $\la$ be the sum of all fundamental weights. Then $V_\la$ sits inside the tensor product of all the fundamental representations
as the Cartan component and one gets the embedding 
$$G/P\subset \prod_{i=1}^{\rk\fg} G/P_i\subset \prod_{i=1}^{\rk\fg} \bP(V_{\om_i}),$$  
where $P_i$ is the maximal parabolic subgroup corresponding to $\om_i$. For classical Lie algebras this embedding can be
described very explicitly in terms of linear algebra, the details are given below.

\subsection{Type A}
Type $A_{N-1}$ refers to the group $SL_N$.
The type $A_{N-1}$ flag variety $F_N$ is defined as the variety of collections $V_1\subset V_2\subset\dots\subset V_{N-1}$
of subspaces of an $N$-dimensional vector space $E$
 such that $\dim V_k=k$ for all $k$. The group $SL_N$ acts transitively on $F_N$.

The coordinate flags (for a fixed basis $e_1,\dots,e_N$) are labeled by the permutation group $S_N$ which is the Weyl group of
for $SL_N$. For $\sigma\in S_N$ the corresponding coordinate flag is given by $V_k={\mathrm{span}} (e_{\sigma(i)})_{i=1}^k$.
In particular, there are $N!$ such flags. 

The variety $F_N$ can be cut into disjoint union of cells (the so called Schubert cells). Each cell contains exactly one 
coordinate flag. The dimension of the cell containing the coordinate flag corresponding to $\sigma\in S_N$ is equal to
the length of $\sigma$ (the number of inversions of $\sigma$).

\subsection{Type B}
Type $B_n$ refers to the orthogonal group $SO_{2n+1}$ consisting of square matrices $g$ of size $(2n+1)$ such that $\det g=1$
and $g$ preserves a nondegenerate symmetric bilinear form (i.e. $(gu,gv)=(u,v)$ for all vectors $u,v$). In what follows
we fix a basis $e_1,\dots,e_{2n+1}$ such that $(e_i,e_j)=\delta_{i+j,2n+2}$.
  
The type $B_n$ flag variety $SOF_{2n+1}$ is defined as the variety of collections $V_1\subset V_2\subset\dots\subset V_n$
of subspaces of a $(2n+1)$-dimensional vector space $E$ such that $\dim V_k=k$ for all $k$ and each $V_k$ is isotropic
(it is enough to say that  $V_n$ is isotropic). 
We have natural embedding $SOF_{2n+1}\subset F_{2n+1}$ defined by
\[
(V_1,\dots,V_n)\mapsto (V_1,\dots,V_n,V_n^\perp,\dots,V_1^\perp).
\]
The group $SO_{2n+1}$ acts transitively on $SOF_{2n+1}$.

The coordinate flags (for our fixed basis $e_1,\dots,e_{2n+1}$ of $E$) are labeled by collections $(I_1,\dots,I_n)$ of subsets of
$\{1,\dots,2n+1\}$ such that $|I_k|=k$, $I_k\subset I_{k+1}$ and $a\in I_n$ implies $2n+2-a\notin I_n$. There are exactly
$2^nn!$ of such coordinates flags and this is the cardinality of the Weyl group of type $B_n$.

The variety $SOF_{2n+1}$ can be cut into disjoint union of cells (the so called Schubert cells). Each cell contains exactly one 
coordinate flag. 

\subsection{Type C}
Type $C_n$ refers to the symplectic group $Sp_{2n}$ consisting of square matrices $g$ of size $2n$ such that $\det g=1$
and $g$ preserves a nondegenerate skew-symmetric bilinear form (i.e. $(gu,gv)=(u,v)$ for all vectors $u,v$). In what follows
we fix a basis $e_1,\dots,e_{2n}$ such that $(e_i,e_j)=\delta_{i+j,2n+1}$ for $i=1,\dots,n$.
  
The type $C_n$ flag variety $SpF_{2n}$ is defined as the variety of collections $V_1\subset V_2\subset\dots\subset V_n$
of subspaces of a $2n$-dimensional vector space $E$ such that $\dim V_k=k$ for all $k$ and each $V_k$ is isotropic
(it is enough to say that  $V_n$ is isotropic). 
We have natural embedding $SpF_{2n}\subset F_{2n}$ defined by
\[
(V_1,\dots,V_n)\mapsto (V_1,\dots,V_n,V_{n-1}^\perp,\dots,V_1^\perp).
\]
The group $Sp_{2n}$ acts transitively on $SpF_{2n}$.

The coordinate flags (for our fixed basis $e_1,\dots,e_{2n}$) are labeled by collections $(I_1,\dots,I_n)$ of subsets of
$\{1,\dots,2n\}$ such that $|I_k|=k$, $I_k\subset I_{k+1}$ and $a\in I_n$ implies $2n+1-a\notin I_n$. There are exactly
$2^nn!$ of such coordinates flags and this is the cardinality of the Weyl group of type $C_n$ (the same as in type $B_n$).
As before, $SpF_{2n}$ is the disjoint union of Schubert cells each containing exactly one coordinate flag.
We note that the varieties $SpF_{2n}$ and $SOF_{2n+1}$ are not isomorphic. 

\subsection{Type D}
Type $D_n$ refers to the orthogonal group $SO_{2n}$ consisting of square matrices $g$ of size $2n$ such that $\det g=1$
and $g$ preserves a nondegenerate symmetric bilinear form. In what follows
we fix a basis $e_1,\dots,e_{2n}$ such that $(e_i,e_j)=\delta_{i+j,2n+1}$.
  
The type $D_n$ flag variety $SOF_{2n}$ is defined as the variety of collections of subspaces of our $2n$-dimensional vector space
$V_1\subset V_2\subset\dots\subset V_n$ such that $\dim V_k=k$ for all $k$ and $V_n$ is isotropic.
However, in contrast with the cases above, there is one extra condition: the dimension of the space
$V_n/(V_n\cap {\rm span} (e_i)_{i=1}^n)$ must be even. If one does not impose such a condition, then the resulting
variety is reducible. The group $SO_{2n}$ acts transitively on $SOF_{2n}$ and as
before we have natural embedding $SOF_{2n}\subset F_{2n}$.

The coordinate flags (for our fixed basis $e_1,\dots,e_{2n}$) are labeled by collections $(I_1,\dots,I_n)$ of subsets of
$\{1,\dots,2n\}$ such that $|I_k|=k$, $I_k\subset I_{k+1}$, $a\in I_n$ implies $2n+2-a\notin I_n$ and 
one extra condition saying that the cardinality of the set $\{a\in I_n, a>n\}$ is even. There are exactly
$2^{n-1}n!$ of such coordinates flags and this is the cardinality of the Weyl group of type $D_n$ and the Euler characteristics 
of $SOF_{2n}$.

\subsection{Odd type C}
Following \cite{M,Pe} (see also \cite{Pr}) we consider a $(2n+1)$-dimensional vector space $E$ and a skew-symmetric 
bilinear form of maximal possible rank (i.e. of rank $2n$).
We fix a basis $e_1,\dots,e_{2n+1}$ such that $(e_i,e_j)=\delta_{i+j,2n+2}$ for $i=1,\dots,n$, $(e_{n+1},v)=0$ for all $v$.
The group of operators preserving this form is called the odd symplectic group $Sp_{2n+1}$. The important difference
with  all the previous cases is that this group is no longer simple and thus does not fit into the Cartan classification.
However, one can still define the corresponding flag variety in the same way as above. Namely,
we consider the flag variety $SpF_{2n+1}$ defined as the variety of collections of subspaces of our $(2n+1)$-dimensional vector space
$V_1\subset V_2\subset\dots\subset V_n$ such that $\dim V_k=k$ for all $k$ and $V_n$ is isotropic.
We note that since our skew-symmetric form has a one-dimensional kernel, the maximal dimension of an isotropic
subspace is not $n$, but $(n+1)$. So one can also study an extended version adding one more isotropic subspace $V_{n+1}$ to our chain.

\section{Degenerate flag varieties of types A and C}
\label{DFV}
In this section we recall basic definitions and results on the degenerate flag varieties of types $A$ and $C$ 
(see \cite{F1,F2,F3,FFL}).

Let $F_s\subset \U(\fn_-)$, $s=0,1,\dots$ be (increasing) Poincar\'e-Birkhoff-Witt filtration on the universal enveloping
algebra. The associated graded algebra is isomorphic to the symmetric algebra $S(\fn_-)$, which can be identified
with the universal enveloping algebra $\U(\fn_-^a)$ of the abelian Lie algebra $\fn_-^a$ whose underlying
space is equal to $\fn_-^a$. The PBW filtration $F_s$ induces the
increasing filtration $F_sv_\la$ on $V_\la$. The associated graded space is denoted by $V_\la^a$, which is naturally a cyclic 
$S(\fn_-)$ module. 

Let $N_-^a=\exp(\fn_-^a)$ be the abelian unipotent group isomorphic to the $\dim \fn_-$ copies of the additive group of the field.
We define the degenerate flag variety $F_\la^a$ as the closure of the orbit $N_-^a v_\la$ inside $\bP(V_\la^a)$.
In types $A$ and $C$ the varieties $F_\la^a$ depend only on the regularity class of $\la$. In particular, $F_\la^a\simeq F_\mu^a$
if both $\la$ and $\mu$ are regular.

It was proved in \cite{F3,FFL} that in types $A$ and $C$ the varieties $F^a_\la$ enjoy very concrete description in terms of linear
algebra. This description fits into the general framework of the quiver Grassmannians theory \cite{CFR1}. 
We give the details below.

\subsection{Type A}
Let us fix a basis $e_1,\dots,e_N$ of a complex vector space $E$ and projections
$pr_k:E\to E$ such that $pr_ke_l=(1-\delta_{k,l})e_l$.
Let ${\rm Gr}(k,N)$ be the Grassmannian of $k$-dimensional subspaces in $E$. 
Then the degenerate flag variety for $SL_N$ is given by
\[
F^a_N=\{(V_1,\dots,V_{N-1}): V_k\in {\rm Gr}(k,N), pr_{k+1}V_k\subset V_{k+1}\}.
\]
The dimension of $F^a_N$ is $N(N-1)/2$ and the Euler characteristic is given by the normalized
median Genocchi number. More precisely, the variety enjoys a cellular decomposition and all the cells are of even (real) dimension.
The cells are labeled by the
collections ${\bf I}=(I_1,\dots,I_{N-1})$ consisting of subsets of the set $\{1,\dots,N\}$
subject to the conditions
\[
|I_k|=k,\ I_k\subset I_{k+1}\cup\{k+1\}.
\]
Each such a collection defines a torus fixed point $p(\bI)\in F^a_N$, whose entries $(V_k)_{k=1}^{N-1}$ are 
${\mathrm{span}}(e_i)_{i\in I_k}$. We denote the cell containing $p(\bI)$ by $C(\bI)$. 
In order to describe $C(\bI)$ we prepare some notation. 
Let us introduce $N-1$ different orderings $>_k$, $k=1,\dots,N~-~1$ on the set $\{1,\dots,N\}$.
Namely, for $1\le k\le N~-~1$ we set
\begin{equation}\label{orderingA}
k>_k k-1 >_k \dots >_k 1>_k N>_k N-1>_k\dots >_k k+1.
\end{equation}
The cell $C(\bI)$ consists of collections of subspaces $(V_1,\dots,V_{N-1})$ such that for all $k$ the space
$V_k$ has a basis of the form (see  \cite{F1})
\begin{equation}\label{celldescrA}
e_{i_a}+\sum_{b<_k i_a, b\notin I_k} x_{a,b}e_b,\ a=1,\dots,k,\ i_a\in I_k,\ x_{a,b}\in\bC. 
\end{equation}
As a consequence, we obtain the following formula for the complex dimension of $C(\bI)$.
Let us call an element $i\in I_k$ initiating, if either $i=k$ or $i\notin I_{k-1}$ (we assume that $I_0=\emptyset$). 
Let $ID_k\subset I_k$ be the set of initiating elements. 
Then one has
\begin{equation}\label{celldimA}
\dim C(\bI) = \sum_{k=1}^{N-1} \sum_{i\in ID_k} |\{j<_k i: j\notin I_k\}|.
\end{equation}

\begin{rem}
Formula \eqref{celldimA} can be easily derived from the explicit description \eqref{celldescrA}
by computing the number of free parameters (taking into account the condition $pr_{k+1}V_k\subset V_{k+1}$). 
\end{rem}

\begin{rem}
Let us consider the equioriented type $A_{N-1}$ quiver and its representation $M=P\oplus I$, which is the direct
sum of all indecomposable projective representations and all indecomposable injective representations (see \cite{CFR1}).
In particular, $\dim M_k=N$, $1\le k\le N-1$ and the maps $M_k\to M_{k+1}$ are given by $pr_{k+1}$ (after 
fixing appropriate bases in $M_k$). Then each $\bI$ gives rise to a subrepresentation $S(\bI)$ of $M$ of dimension 
$(1,\dots,N~-~1~)$: the $S(\bI)_k$ is spanned by $e_a$, $a\in I_k$.   
Then the initiating elements are in one-to-one correspondence with the indecomposable summands of $S(\bI)$.
\end{rem}

Let us note that all the results above generalize to the case of partial flag varieties.
Namely, given a collection $\bd=(d_1,\dots,d_s)$ with $1\le d_1<\dots<d_s\le N-1$ one defines the partial degenerate
flag variety $F^a_\bd$ as the projection of $F^a_N$ to the product $\prod_{i=1}^s {\rm Gr}(d_i,N)$. One easily sees that
each cell $C(\bI)$ in  $F^a_N$ projects to a  cell $C(\bI_\bd)$ in $F^a_\bd$ thus providing a cellular decomposition of $F^a_\bd$
(here $\bI_\bd=(I_{d_1},\dots,I_{d_s})$). 
The dimension of the cell  $C(\bI_\bd)$ is computed by the formula \eqref{celldimA} with the only
reservation that the sum in the right hand side of \eqref{celldimA} is taken over $k\in\{d_1,\dots,d_s\}$.

Finally, recall~\cite{F1} that a convenient way to parametrize the cells of $F_N^a$ goes through the Dellac 
configurations $DC_N$~\cite{De}. A Dellac configuration  $\mathcal{D} \in DC_N$ is a tableau made of $N$ columns and $2N$ rows, 
containing $2N$ points such that~:
\begin{itemize}
\item every row contains exactly one point;
\item every column contains exactly two points;
\item if there is a point in the box $(j,i)$ of $\mathcal{D}$ (\textit{i.e.}, in its $j$-th column from left to right 
and its $i$-th row from bottom to top), then $j \leq i \leq N+j$.
\end{itemize}
If a box $(j,i)$ of a configuration contains a point $p$, we say that $p = (j,i)$. In \cite{F1}, the author proved 
that if $\mathcal{D} \in DC_N$ corresponds to a given cell $C$ of $F_N^a$, then $\dim(C)$ is the number of \textit{inversions} 
of $\mathcal{D}$, \textit{i.e.}, the number of pairs $((j,i),(j',i'))$ of points of $\mathcal{D}$ such that $j < j'$ and $i > i'$. 
For example, in Figure \ref{fig:DC3}, we depict the 7 elements of $DC_3$, whose inversions are represented by segments, computing 
the Poincar\'e polynomial $P_{F_3^a}(q) = 1+2q+3q^2+q^3$.

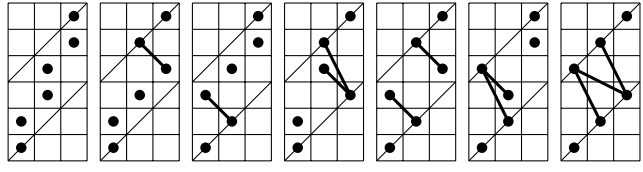
\begin{figure}[!htbp]

\begin{center}

\begin{tikzpicture}[scale=0.35]
\draw (0,0) grid[step=1] (3,6);
\draw (0,0) -- (3,3);
\draw (0,3) -- (3,6);
\fill (0.5,0.5) circle (0.2);
\fill (0.5,1.5) circle (0.2);
\fill (1.5,2.5) circle (0.2);
\fill (1.5,3.5) circle (0.2);
\fill (2.5,4.5) circle (0.2);
\fill (2.5,5.5) circle (0.2);

\begin{scope}[xshift=3.5cm]
\draw (0,0) grid[step=1] (3,6);
\draw (0,0) -- (3,3);
\draw (0,3) -- (3,6);
\fill (0.5,0.5) circle (0.2);
\fill (0.5,1.5) circle (0.2);
\fill (1.5,2.5) circle (0.2);
\fill (2.5,3.5) circle (0.2);
\fill (1.5,4.5) circle (0.2);
\fill (2.5,5.5) circle (0.2);

\draw[line width=0.4mm,color=black] (1.5,4.5) -- (2.5,3.5);
\end{scope}

\begin{scope}[xshift=7cm]
\draw (0,0) grid[step=1] (3,6);
\draw (0,0) -- (3,3);
\draw (0,3) -- (3,6);
\fill (0.5,0.5) circle (0.2);
\fill (1.5,1.5) circle (0.2);
\fill (0.5,2.5) circle (0.2);
\fill (1.5,3.5) circle (0.2);
\fill (2.5,4.5) circle (0.2);
\fill (2.5,5.5) circle (0.2);

\draw[line width=0.4mm,color=black] (0.5,2.5) -- (1.5,1.5);
\end{scope}

\begin{scope}[xshift=10.5cm]
\draw (0,0) grid[step=1] (3,6);
\draw (0,0) -- (3,3);
\draw (0,3) -- (3,6);
\fill (0.5,0.5) circle (0.2);
\fill (0.5,1.5) circle (0.2);
\fill (2.5,2.5) circle (0.2);
\fill (1.5,3.5) circle (0.2);
\fill (1.5,4.5) circle (0.2);
\fill (2.5,5.5) circle (0.2);

\draw[line width=0.4mm,color=black] (1.5,4.5) -- (2.5,2.5);
\draw[line width=0.4mm,color=black] (1.5,3.5) -- (2.5,2.5);
\end{scope}

\begin{scope}[xshift=14cm]
\draw (0,0) grid[step=1] (3,6);
\draw (0,0) -- (3,3);
\draw (0,3) -- (3,6);
\fill (0.5,0.5) circle (0.2);
\fill (1.5,1.5) circle (0.2);
\fill (0.5,2.5) circle (0.2);
\fill (2.5,3.5) circle (0.2);
\fill (1.5,4.5) circle (0.2);
\fill (2.5,5.5) circle (0.2);

\draw[line width=0.4mm,color=black] (0.5,2.5) -- (1.5,1.5);
\draw[line width=0.4mm,color=black] (1.5,4.5) -- (2.5,3.5);
\end{scope}

\begin{scope}[xshift=17.5cm]
\draw (0,0) grid[step=1] (3,6);
\draw (0,0) -- (3,3);
\draw (0,3) -- (3,6);
\fill (0.5,0.5) circle (0.2);
\fill (1.5,1.5) circle (0.2);
\fill (1.5,2.5) circle (0.2);
\fill (0.5,3.5) circle (0.2);
\fill (2.5,4.5) circle (0.2);
\fill (2.5,5.5) circle (0.2);

\draw[line width=0.4mm,color=black] (0.5,3.5) -- (1.5,2.5);
\draw[line width=0.4mm,color=black] (0.5,3.5) -- (1.5,1.5);
\end{scope}

\begin{scope}[xshift=21cm]
\draw (0,0) grid[step=1] (3,6);
\draw (0,0) -- (3,3);
\draw (0,3) -- (3,6);
\fill (0.5,0.5) circle (0.2);
\fill (1.5,1.5) circle (0.2);
\fill (2.5,2.5) circle (0.2);
\fill (0.5,3.5) circle (0.2);
\fill (1.5,4.5) circle (0.2);
\fill (2.5,5.5) circle (0.2);

\draw[line width=0.4mm,color=black] (0.5,3.5) -- (2.5,2.5);
\draw[line width=0.4mm,color=black] (0.5,3.5) -- (1.5,1.5);
\draw[line width=0.4mm,color=black] (1.5,4.5) -- (2.5,2.5);
\end{scope}
\end{tikzpicture}

\end{center}

\caption{The $7$ elements of $DC_3$.}
\label{fig:DC3}

\end{figure}

\subsection{Type C}
Now let us define the symplectic degenerate flag variety $SpF^a_{2n}$. 
We fix a symplectic (skew-symmetric nondegenerate) form on $E=\bC^{2n}$ by setting $(e_k,e_l)=\delta_{k+l,2n+1}$, $1\le k\le n$.
Then the symplectic degenerate flag variety $SpF^a_{2n}$ has the following explicit realization (see \cite{FFL}):
\[
SpF^a_{2n}=\{(V_1,\dots,V_n):\ \dim V_k=k, pr_{k+1}V_k\subset V_{k+1}, V_n \text{ is Lagrangian}\}. 
\]
We note that $\dim SpF^a_{2n}=n^2$ and the Euler characteristic is given by the number of symplectic
Dellac configurations $S \in SDC_{2n}$ defined in~\cite{FF}. These are elements $S \in DC_{2n}$ such that $R(S) = S$, where $R$ is 
the central reflection with respect to the center of $S$. 
In other words $SDC_{2n}$ is the subset of $DC_{2n}$ made of the elements 
that contain a point in a box $(j,i) \in \{1,\dots,2n\} \times \{1,\dots,4n\}$ if and only if the box $(2n+1-j,4n+1-i)$ also 
contains a point. In what follows we refer to the elements of $SDC_{2n}$ as symmetric Dellac configurations
(since we will also need the Dellac configurations with odd number of columns symmetric with respect to the center). 

\begin{rem}\label{CtoA}
We have two natural embeddings of $SpF^a_{2n}$ into the type $A$ degenerate flag varieties. 
First, one can embed  $SpF^a_{2n}$ into $F^a_{2n}$ by the map
\[
(V_1,\dots,V_n)\mapsto (V_1,\dots,V_n,V_{n-1}^\perp,\dots, V_1^\perp)
\]
(note that $V_n^\perp=V_n$). Second, $SpF^a_{2n}$ obviously sits inside
$F^a_{\bd}$ for $\bd=(1,\dots,n)$.
\end{rem}

\subsection{Resolution of singularities}\label{resofsing}
Let $R_N$ be the variety of collections $(V_{i,j})_{1\le i\le j\le N-1}$ of subspaces of a vector space $E$ with 
basis $e_1,\dots,e_N$ such that 
\begin{itemize}
\item $\dim V_{i,j}=i$,\ $V_{i,j}\subset {\rm span} (e_1,\dots,e_i,e_{j+1},\dots,e_N)$,
\item $V_{i,j}\subset V_{i+1,j}$,\ $pr_{i+1} V_{i,j}\subset V_{i,j+1}$.
\end{itemize}
Then one can show (see \cite{FFi}) that $R_N$ is a tower of $\bP^1$ fibrations (and thus smooth) of dimension 
$N(N-1)/2$. The map $R_N\to F^a_N$ sending a collection $(V_{i,j})\in R_N$ to $(V_{i,i})_{i=1}^{N-1}$
is a resolution of singularities.

Similar construction works in the symplectic case \cite{FFL}. We fix a symplectic form on a $2n$-dimensional vector
space  $E={\rm span} (e_1,\dots,e_{2n})$ such that $(e_i,e_{2n+1-i})=1$ for $i=1,\dots,n$. 
The symplectic resolution $SpR_{2n}$
consists of collections $(V_{i,j})$, $1\le i\le n$, $i+j\le 2n$ of subspaces of $E$ such that 
\begin{itemize}
\item $\dim V_{i,j}=i$,\ $V_{i,j}\subset {\rm span} (e_1,\dots,e_i,e_{j+1},\dots,e_{2n})$,
\item $V_{i,j}\subset V_{i+1,j}$,\ $pr_{i+1} V_{i,j}\subset V_{i,j+1}$,
\item $V_{i,2n+1-i}$ is Lagrangian in ${\rm span} (e_1,\dots,e_i,e_{2n+1-i},\dots,e_{2n})$.
\end{itemize}
As in type $A$, $SpR_{2n}$ is a resolution of singularities of $SpF^a_{2n}$ via the map 
$(V_{i,j})\mapsto (V_{i,i})_{i=1}^n$.

\section{Degenerate flags: even symplectic case}\label{ESC}

\subsection{Dimension of cells: algorithm}
Let us consider a collection ${\bf I}=(I_1,\dots,I_n)$ consisting of subsets of the set $\{1,\dots,2n\}$
subject to the conditions
\begin{equation}\label{CI}
|I_k|=k,\ I_k\subset I_{k+1}\cup\{k+1\},\ a\in I_n \text{ implies } 2n+1-a\notin I_n.
\end{equation}
Let $\mathcal{I}_{2n}$ be the set of these collections. Each such a collection corresponds to a point $p(\bI)$ 
in the degenerate symplectic flag variety defined by $p(\bI)_k={\rm span}(e_i, i\in I_k)$. 
Recall (see Remark \ref{CtoA}) that $SpF^a_{2n}$ sits inside $F^a_{(1,\dots,n)}$.
In what follows we denote the collection $(1,\dots,n)$ by $\bd$.

\begin{lem}\label{lem:ac}
Assume that a collection $\bI$ satisfies \eqref{CI} (i.e. $\bI\in \mathcal{I}_{2n}$) and let $C_A(\bI) \subset F^a_\bd$ be
the corresponding cell. Then the intersection $C(\bI)=SpF^a_{2n}\cap C_A(\bI)$ is an affine cell. 
\end{lem}
\begin{proof}
Let $\bI=(I_1,\dots,I_n)$. Let $\bar I_n$ be the complement to $I_n$ inside the set $\{1,\dots,2n\}$. In particular,
$\bar I_n$ contains $n$ elements and $a\in I_n$ if and only if $2n+1-i\in \bar I_n$. Let $C_A(\bI)$ be the cell 
inside $F^a_\bd$ corresponding to the collection $\bI$. In what follows we use the notation 
\begin{equation}
C(\bI)=SpF^a_{2n}\cap C_A(\bI).
\end{equation} 
We note that $C(\bI)$ consists of points $(V_1,\dots,V_n)\in C_A(\bI)$ such that  $V_n$ is Lagrangian.
 
Recall the ordering $<_n$ on the set $\{1,\dots,2n\}$ (see \eqref{orderingA}). Then $V_n$ 
shows up as the last component of a point in $C_A(\bI)$ if and only if $V_n$ is the linear span of the vectors
\[
l_a=e_a+\sum_{b<_n a, b\in\bar I_n} x_{a,b} e_b,\ a\in I_n. 
\]   
We note that the coordinates $x_{a,b}$ form a part of the coordinate system on the affine cell $C_A(\bI)$.
Now $V_n$ is Lagrangian if and only if $(l_{a_1}, l_{a_2})=0$ for all $a,b\in I_n$, which is equivalent to 
the system of linear conditions labeled by all pairs $a_1,a_2\in I_n$ such that $a_1<_n a_2$:
\begin{equation}
x_{a_1,2n+1-a_2}\pm x_{a_2,2n+1-a_1}=0,\ 2n+1-a_1<_n a_2
\end{equation}
(recall that our symplectic form pairs nontrivially $e_a$ and $e_{2n+1-a}$). Hence  $C(\bI)$ is cut out inside 
$C_A(\bI)$ by linear equations and hence is isomorphic to an affine space.
\end{proof}

The proof above allows to compare the dimensions of $C_A(\bI)$ and that of $C(\bI)$. 
\begin{cor}\label{comparedimensions}
Let $\bI\in \mathcal{I}_{2n}$. Then
\[
\dim C_A(\bI) - \dim C(\bI) = |\{(a_1,a_2)\in I_n:\ a_1<_n a_2, 2n+1-a_1<_n a_2\}|.
\]
\end{cor}

Now let us give an explicit (inductive) way to compute the dimension $d(\bI)$ of the cell $C(\bI)$.
We will use $n$ different orderings $>_k$, $k=1,\dots,n$ on the set $\{1,\dots,2n\}$ (see \eqref{orderingA}).
We represent $\bI$ graphically as follows (which reflects the latter orderings)~: first we draw a parallelogram made 
of $2n \times n$ horizontally linked white disks as depicted in Figure \ref{fig:skeletton};

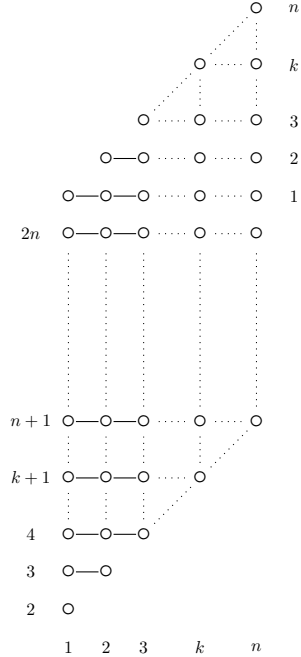
\begin{figure}[h]
\begin{tikzpicture}[scale=0.5]

\draw (1,2) node[scale=1]{$\circ$};
\draw (1,3) node[scale=1]{$\circ$};
\draw (1,4) node[scale=1]{$\circ$};
\draw [dotted] (1,4) [shorten <= 0.2cm, shorten >= 0.2cm] -- (1,5.5);
\draw (1,5.5) node[scale=1]{$\circ$};
\draw [dotted] (1,5.5) [shorten <= 0.2cm, shorten >= 0.2cm] -- (1,7);
\draw (1,7) node[scale=1]{$\circ$};
\draw [dotted] (1,7) [shorten <= 0.2cm, shorten >= 0.2cm] -- (1,12);
\draw (1,12) node[scale=1]{$\circ$};
\draw (1,13) node[scale=1]{$\circ$};

\draw (2,3) node[scale=1]{$\circ$};
\draw (2,4) node[scale=1]{$\circ$};
\draw [dotted] (2,4) [shorten <= 0.2cm, shorten >= 0.2cm] -- (2,5.5);
\draw (2,5.5) node[scale=1]{$\circ$};
\draw [dotted] (2,5.5) [shorten <= 0.2cm, shorten >= 0.2cm] -- (2,7);
\draw (2,7) node[scale=1]{$\circ$};
\draw [dotted] (2,7) [shorten <= 0.2cm, shorten >= 0.2cm] -- (2,12);
\draw (2,12) node[scale=1]{$\circ$};
\draw (2,13) node[scale=1]{$\circ$};
\draw (2,14) node[scale=1]{$\circ$};

\draw (3,4) node[scale=1]{$\circ$};
\draw [dotted] (3,4) [shorten <= 0.2cm, shorten >= 0.2cm] -- (3,5.5);
\draw (3,5.5) node[scale=1]{$\circ$};
\draw [dotted] (3,5.5) [shorten <= 0.2cm, shorten >= 0.2cm] -- (3,7);
\draw (3,7) node[scale=1]{$\circ$};
\draw [dotted] (3,7) [shorten <= 0.2cm, shorten >= 0.2cm] -- (3,12);
\draw (3,12) node[scale=1]{$\circ$};
\draw (3,13) node[scale=1]{$\circ$};
\draw (3,14) node[scale=1]{$\circ$};
\draw (3,15) node[scale=1]{$\circ$};

\draw (4.5,5.5) node[scale=1]{$\circ$};
\draw [dotted] (4.5,7) [shorten <= 0.2cm, shorten >= 0.2cm] -- (4.5,12);
\draw (4.5,7) node[scale=1]{$\circ$};
\draw (4.5,12) node[scale=1]{$\circ$};
\draw (4.5,13) node[scale=1]{$\circ$};
\draw (4.5,14) node[scale=1]{$\circ$};
\draw (4.5,15) node[scale=1]{$\circ$};
\draw (4.5,16.5) node[scale=1]{$\circ$};

\draw (6,7) node[scale=1]{$\circ$};
\draw [dotted] (6,7) [shorten <= 0.2cm, shorten >= 0.2cm] -- (6,12);
\draw (6,12) node[scale=1]{$\circ$};
\draw (6,13) node[scale=1]{$\circ$};
\draw (6,14) node[scale=1]{$\circ$};
\draw (6,15) node[scale=1]{$\circ$};
\draw (6,16.5) node[scale=1]{$\circ$};
\draw (6,18) node[scale=1]{$\circ$};

\draw [dotted] (3,4) [shorten <= 0.2cm, shorten >= 0.15cm] -- (4.5,5.5);
\draw [dotted] (4.5,5.5) [shorten <= 0.2cm, shorten >= 0.15cm] -- (6,7);
\draw [dotted] (4.5,5.5) [shorten <= 0.2cm, shorten >= 0.15cm] -- (4.5,7);
\draw [dotted] (3,15) [shorten <= 0.2cm, shorten >= 0.15cm] -- (4.5,16.5);
\draw [dotted] (4.5,16.5) [shorten <= 0.2cm, shorten >= 0.15cm] -- (6,18);
\draw [dotted] (4.5,15) [shorten <= 0.2cm, shorten >= 0.2cm] -- (4.5,16.5);
\draw [dotted] (3,7) [shorten <= 0.2cm, shorten >= 0.2cm] -- (4.5,7);
\draw [dotted] (3,5.5) [shorten <= 0.2cm, shorten >= 0.2cm] -- (4.5,5.5);
\draw [dotted] (4.5,7) [shorten <= 0.2cm, shorten >= 0.2cm] -- (6,7);
\draw [dotted] (3,12) [shorten <= 0.2cm, shorten >= 0.2cm] -- (4.5,12);
\draw [dotted] (3,13) [shorten <= 0.2cm, shorten >= 0.2cm] -- (4.5,13);
\draw [dotted] (3,14) [shorten <= 0.2cm, shorten >= 0.2cm] -- (4.5,14);
\draw [dotted] (3,15) [shorten <= 0.2cm, shorten >= 0.2cm] -- (4.5,15);
\draw [dotted] (4.5,12) [shorten <= 0.2cm, shorten >= 0.2cm] -- (6,12);
\draw [dotted] (4.5,13) [shorten <= 0.2cm, shorten >= 0.2cm] -- (6,13);
\draw [dotted] (4.5,14) [shorten <= 0.2cm, shorten >= 0.2cm] -- (6,14);
\draw [dotted] (4.5,15) [shorten <= 0.2cm, shorten >= 0.2cm] -- (6,15);
\draw [dotted] (4.5,16.5) [shorten <= 0.2cm, shorten >= 0.2cm] -- (6,16.5);
\draw [dotted] (6,15) [shorten <= 0.2cm, shorten >= 0.2cm] -- (6,16.5);
\draw [dotted] (6,16.5) [shorten <= 0.2cm, shorten >= 0.2cm] -- (6,18);

\draw (1,3) [shorten <= 0.1cm, shorten >= 0.1cm] -- (2,3);
\draw (1,4) [shorten <= 0.1cm, shorten >= 0.1cm] -- (2,4);
\draw (2,4) [shorten <= 0.1cm, shorten >= 0.1cm] -- (3,4);
\draw (1,5.5) [shorten <= 0.1cm, shorten >= 0.1cm] -- (2,5.5);
\draw (2,5.5) [shorten <= 0.1cm, shorten >= 0.1cm] -- (3,5.5);
\draw (1,7) [shorten <= 0.1cm, shorten >= 0.1cm] -- (2,7);
\draw (2,7) [shorten <= 0.1cm, shorten >= 0.1cm] -- (3,7);
\draw (1,12) [shorten <= 0.1cm, shorten >= 0.1cm] -- (2,12);
\draw (2,12) [shorten <= 0.1cm, shorten >= 0.1cm] -- (3,12);
\draw (1,13) [shorten <= 0.1cm, shorten >= 0.1cm] -- (2,13);
\draw (2,13) [shorten <= 0.1cm, shorten >= 0.1cm] -- (3,13);
\draw (2,14) [shorten <= 0.1cm, shorten >= 0.1cm] -- (3,14);

\draw (0,2) node[scale=0.6]{$2$};
\draw (0,3) node[scale=0.6]{$3$};
\draw (0,4) node[scale=0.6]{$4$};
\draw (0,5.5) node[scale=0.6]{$k+1$};
\draw (0,7) node[scale=0.6]{$n+1$};
\draw (0,12) node[scale=0.6]{$2n$};
\draw (7,13) node[scale=0.6]{$1$};
\draw (7,14) node[scale=0.6]{$2$};
\draw (7,15) node[scale=0.6]{$3$};
\draw (7,16.5) node[scale=0.6]{$k$};
\draw (7,18) node[scale=0.6]{$n$};

\draw (1,1) node[scale=0.6]{$1$};
\draw (2,1) node[scale=0.6]{$2$};
\draw (3,1) node[scale=0.6]{$3$};
\draw (4.5,1) node[scale=0.6]{$k$};
\draw (6,1) node[scale=0.6]{$n$};

\end{tikzpicture}
\caption{Skeletton of a graphical representation $\mathcal{P}(\bI)$.}
\label{fig:skeletton}
\end{figure}
if a disk $D$ is located in the $k$-th column (from left to right) and in the row labeled with the integer 
$l \in \{k+1,\hdots,2n,1,\hdots,k\}$, we say that $D = [k:l]$. We then define $\mathcal{P}(\bI)$ as this parallelogram in which 
every disk $[k:l]$ with $l \in I_k$ is painted in black.

For example, consider the collection $$\bI_0 = (\{3\},\{3,7\},\{1,4,7\},\{1,4,6,7\}) \in \mathcal{I}_8,$$
then $\mathcal{P}(\bI_0)$ is as depicted in Figure \ref{fig:PI0}.

\begin{figure}[h]

\begin{tikzpicture}[scale=0.5]

\draw (1,2) node[scale=1]{$\circ$};
\draw (1,3) node[scale=1]{$\bullet$};
\draw (1,4) node[scale=1]{$\circ$};
\draw (1,5) node[scale=1]{$\circ$};
\draw (1,6) node[scale=1]{$\circ$};
\draw (1,7) node[scale=1]{$\circ$};
\draw (1,8) node[scale=1]{$\circ$};
\draw (1,9) node[scale=1]{$\circ$};

\draw (2,3) node[scale=1]{$\bullet$};
\draw (2,4) node[scale=1]{$\circ$};
\draw (2,5) node[scale=1]{$\circ$};
\draw (2,6) node[scale=1]{$\circ$};
\draw (2,7) node[scale=1]{$\bullet$};
\draw (2,8) node[scale=1]{$\circ$};
\draw (2,9) node[scale=1]{$\circ$};
\draw (2,10) node[scale=1]{$\circ$};

\draw (3,4) node[scale=1]{$\bullet$};
\draw (3,5) node[scale=1]{$\circ$};
\draw (3,6) node[scale=1]{$\circ$};
\draw (3,7) node[scale=1]{$\bullet$};
\draw (3,8) node[scale=1]{$\circ$};
\draw (3,9) node[scale=1]{$\bullet$};
\draw (3,10) node[scale=1]{$\circ$};
\draw (3,11) node[scale=1]{$\circ$};

\draw (4,5) node[scale=1]{$\circ$};
\draw (4,6) node[scale=1]{$\bullet$};
\draw (4,7) node[scale=1]{$\bullet$};
\draw (4,8) node[scale=1]{$\circ$};
\draw (4,9) node[scale=1]{$\bullet$};
\draw (4,10) node[scale=1]{$\circ$};
\draw (4,11) node[scale=1]{$\circ$};
\draw (4,12) node[scale=1]{$\bullet$};

\draw (1,3) [shorten <= 0.1cm, shorten >= 0.1cm] -- (2,3);
\draw (1,4) [shorten <= 0.1cm, shorten >= 0.1cm] -- (2,4);
\draw (2,4) [shorten <= 0.1cm, shorten >= 0.1cm] -- (3,4);
\draw (1,5) [shorten <= 0.1cm, shorten >= 0.1cm] -- (2,5);
\draw (2,5) [shorten <= 0.1cm, shorten >= 0.1cm] -- (3,5);
\draw (3,5) [shorten <= 0.1cm, shorten >= 0.1cm] -- (4,5);
\draw (1,6) [shorten <= 0.1cm, shorten >= 0.1cm] -- (2,6);
\draw (2,6) [shorten <= 0.1cm, shorten >= 0.1cm] -- (3,6);
\draw (3,6) [shorten <= 0.1cm, shorten >= 0.1cm] -- (4,6);
\draw (1,7) [shorten <= 0.1cm, shorten >= 0.1cm] -- (2,7);
\draw (2,7) [shorten <= 0.1cm, shorten >= 0.1cm] -- (3,7);
\draw (3,7) [shorten <= 0.1cm, shorten >= 0.1cm] -- (4,7);
\draw (1,8) [shorten <= 0.1cm, shorten >= 0.1cm] -- (2,8);
\draw (2,8) [shorten <= 0.1cm, shorten >= 0.1cm] -- (3,8);
\draw (3,8) [shorten <= 0.1cm, shorten >= 0.1cm] -- (4,8);
\draw (1,9) [shorten <= 0.1cm, shorten >= 0.1cm] -- (2,9);
\draw (2,9) [shorten <= 0.1cm, shorten >= 0.1cm] -- (3,9);
\draw (3,9) [shorten <= 0.1cm, shorten >= 0.1cm] -- (4,9);
\draw (2,10) [shorten <= 0.1cm, shorten >= 0.1cm] -- (3,10);
\draw (3,10) [shorten <= 0.1cm, shorten >= 0.1cm] -- (4,10);
\draw (3,11) [shorten <= 0.1cm, shorten >= 0.1cm] -- (4,11);

\draw (0,2) node[scale=0.6]{$2$};
\draw (0,3) node[scale=0.6]{$3$};
\draw (0,4) node[scale=0.6]{$4$};
\draw (0,5) node[scale=0.6]{$5$};
\draw (0,6) node[scale=0.6]{$6$};
\draw (0,7) node[scale=0.6]{$7$};
\draw (0,8) node[scale=0.6]{$8$};

\draw (5,9) node[scale=0.6]{$1$};
\draw (5,10) node[scale=0.6]{$2$};
\draw (5,11) node[scale=0.6]{$3$};
\draw (5,12) node[scale=0.6]{$4$};

\draw (1,1) node[scale=0.6]{$1$};
\draw (2,1) node[scale=0.6]{$2$};
\draw (3,1) node[scale=0.6]{$3$};
\draw (4,1) node[scale=0.6]{$4$};

\end{tikzpicture}
\caption{Graph $\mathcal{P}(\bI_0)$.}
\label{fig:PI0}
\end{figure}
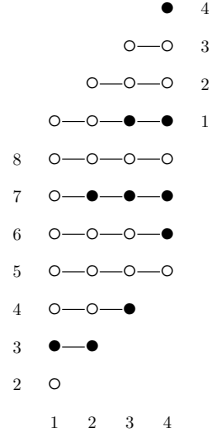
Note that in general, for all $k \in \{1,2,\hdots,n\}$, the $k$-th column $C_k(\bI)$ (from left to right) of $\mathcal{P}(\bI)$ contains
 $|I_k| = k$ black disks, and that if a disk is black, then all the disks located on the right of it in the same row are also black. 
When a black disk is the first (from left to right) of its row, we say it is an \textit{initiating disk}. Let $ID_k(\bI)$ be the set 
of initiating disks of $C_k(\bI)$. For all $D = [k:l] \in ID_k(\bI)$, the number of white disks of $C_k(\bI)$ located under $D$ is 
denoted by $w(D) = w(k:l)$ ; also, the set of black disks of $C_k(\bI)$ located under $D$ is denoted by $B(D)$.

Now, for $1\le k\le n$ and $a,b\in\{1,\dots,2n\}$, $a+b\ne 2n+1$ we define $s(a,b,k)$ as follows:
\[
s(a,b,k)=
\begin{cases}
1, & a,b\ge_k 2n-k+1 \text{ and } 2n+1-b<_k a,\\
0, & \text{ otherwise}.
\end{cases}
\]
\begin{rem}
$s(a,b,k)$ is either one or zero. It is equal to $1$ if both $a$ and $b$ are from the set 
$\{k,k-1,\dots,1,2n,2n-1,\dots,2n-k+1\}$ and $2n+1-b<_k a$ (note that if $b\ge_k2n-k+1$, then
$2n+1-b\ge_k 2n-k+1$). 
\end{rem}

\begin{rem}
Corollary \ref{comparedimensions} can be rephrased in terms of the function $s$ :
let $d_A(\bI)=\dim C_A(\bI)$, then 
$d(\bI)=d_A(\bI)-\sum_{1\le a<_n b\le n} s(a,b,n)$.
\end{rem}

Now let us describe the inductive algorithm for computing the dimensions $d(\bI)=\dim C(\bI)$.
We do it in $n$ steps defining numbers $d_1\le \dots \le d_n=d(\bI)$ adding a nonnegative integer at each step.

We start with $d_1$. Let $I_1=\{i\}$. We define $d_1$ as the number of elements $x\in\{1,\dots,2n\}$ such that 
$x<_1 i$.

Now let us define $d_2$. Let $I_1=\{i\}$. We consider two cases. First, assume $2\notin I_1$. Let $I_2\setminus I_1=\{j\}$.  
We define
\[
d_2=d_1+|\{x: x<_2 j, x\notin I_1\}| - s(j,i,2).
\] 
Second, let $2\in I_1$. Let $I_2=\{j_1,j_2\}$, $j_1>_2 j_2$. Then
\[
d_2=d_1+|\{x: x<_2 j_2\}|+ |\{x: x<_2 j_1\}|-1- s(j_1,j_2,2).
\] 

Now assume $d_{k-1}$ is already defined. Let us define $d_k$.  
We consider two cases. First, assume $k\notin I_{k-1}$. Let $I_k\setminus I_{k-1}=\{j\}$.  
We define
\[
d_k=d_{k-1}+|\{x: x<_k j, x\notin I_{k-1}\}| - \sum_{i\in I_{k-1}} s(j,i,k).
\] 
Second, let $k\in I_{k-1}$ and let $I_{k-1}'=I_{k-1}\setminus\{k\}$.  
Let $I_k\setminus I'_{k-1}=\{j_1,j_2\}$, $j_1>_k j_2$. Then
\begin{multline*}
d_k=d_{k-1}+|\{x: x<_k j_2, x\notin I_{k-1}'\}|+ |\{x: x<_k j_1, x\notin I_{k-1}'\}|-1\\
-s(j_1,j_2,k)- \sum_{i\in I_{k-1}'}  (s(j_1,i,k)+s(j_2,i,k)).
\end{multline*}

In terms of the graphical representation $\mathcal{P}(\bI)$, we obtain
\begin{equation}
\label{eq:inductiondk}
d_k = d_{k-1} + \sum_{D=[k:l] \in ID_k(\bI)} \left( w(D) - \sum_{[k:l'] \in B(D)} s(l,l',k)\right)
\end{equation}
for all $k \in \{1,\hdots,n\}$, where $d_0$ is defined as $0$. For the example $\bI = \bI_0 \in \mathcal{I}_8$ of Figure \ref{fig:PI0}, we obtain~:
\begin{itemize}
\item $ID_1(\bI_0) = \{[1:3]\}$ and $d_1 = w(1:3) = 1$;
\item $ID_2(\bI_0) = \{[2:7]\}$ and $$d_2 = d_1 + (w(2:7) - s(7,3)) = 1 + (3 -0) = 4;$$
\item $ID_3(\bI_0) = \{[3:4],[3:1]\}$ and \begin{align*}d_3 &= d_2 + w(3:4) + (w(3:1)-s(1,7)-s(1,4)) \\ &= 4 + 0 + (3-0-0) = 7;\end{align*}
\item $ID_4(\bI_0) = \{[4:6],[4:4]\}$ and \begin{align*}d_4 &= d_3 + w(4:6) + (w(4:4)-s(4,1)-s(4,7)-s(4,6))\\
 & = 7 + 1 + (4-1-1-1) = 9.\end{align*}
\end{itemize}

\begin{thm}\label{thm:algorithm}
For all $\bI\in \mathcal{I}_{2n}$ the number
$d_n$ is equal to the dimension of $C(\bI)$. 
\end{thm}
\begin{proof}
Let $C_A(\bI)\subset F^a_\bd$ be the cell corresponding to 
$\bI$. Let $t_k: F^a_\bd\to\prod_{i=1}^k {\rm Gr}(k,2n)$ be the projecton to the first $k$ components, i..e.
\[
t_k(V_1,\dots,V_n)=(V_1,\dots,V_k).
\]  
By construction (see \eqref{celldescrA}) the image $t_k C_A(\bI)$ is an affine space. 
We prove by induction on $k$ 
that for any $1\le k\le n$ one has $\dim t_k C(\bI)=d_k$ (recall $C_A(\bI)\supset C(\bI)$). Note that the 
$k=n$ case is exactly the claim of our theorem.

The base of induction $k=1$ is clear, since $t_1 C_A(\bI)=t_1 C(\bI)$ is an affine space of dimension $d_1$.
Now assume the claim is proved for some $k-1<n$. We consider two cases.
First, assume $k\notin I_{k-1}$. Let $I_k\setminus I_{k-1}=\{j\}$.  
Then 
\[
\dim t_k C_A(\bI)=\dim t_{k-1} C_A(\bI) + |\{b: b<_k j, b\notin I_{k-1}\}|,
\] 
since $\dim pr_k V_{k-1}=k-1$ and in order to determine $V_k$ we need to add one more vector of the form
\[
l=e_j+\sum_{b<_k j, b\notin I_k} x_{j,b} e_b.
\]
Now  in order to ensure that the resulting point $(V_1,\dots,V_k)$ belongs to $t_k C_A(\bI)$
we need to impose $\sum_{i\in I_{k-1}} s(j,i,k)$ linear relations on the coefficients $x_{j,b}$.
Indeed $f_kV_k$ is Lagrangian in $W_k$ if and only if the vector $f_kl$ is orthogonal to the space $f_k pr_k V_{k-1}$.
Recall that $V_{k-1}$ has a basis labeled by $i\in I_{k-1}$ of the form \eqref{celldescrA}. For $i \in I_{k-1}$, consider the product
\begin{equation*}
\pi = \left(f_k\left(e_j+\sum_{b<_k j, b\notin I_k} x_{j,b} e_b\right), f_k\left(e_i+\sum_{c<_k i, c\notin I_k} x_{i,c} e_c\right)\right).
\end{equation*}
Since $f_kV_k$ is Lagrangian in $W_k$, we have $\pi = 0$. Let us now expand $\pi$ in the products $(e_b,e_c)$. If $s(j,i,k) = 0$, then all of them equal $0$. Otherwise, we obtain $0 = \pi = x_{j,2n+1-i}(e_{2n+1-i},e_i)+x_{i,2n+1-j}(e_j,e_{2n+1-j})$, hence an equation of the kind $x_{j,2n+1-i} = \pm x_{i,2n+1-j}$. Therefore,
\[
\dim t_k C(\bI)=\dim t_{k-1} C(\bI) + |\{b: b<_k j, b\notin I_{k-1}\}|-\sum_{i\in I_{k-1}} s(j,i,k),
\]
which is equal to $d_k$.
The second case $k\in I_{k-1}$ is proved in the same way, by showing that
and if $I_k\setminus (I_{k-1}\setminus\{k\})=\{j_1,j_2\}$, $j_1>_k j_2$, then 
\[
\dim t_k C_A(\bI)= |\{x: x<_k j_2, x\notin I_{k-1}'\}|+ |\{x: x<_k j_1, x\notin I_{k-1}'\}|-1
\]
and 
\[
\dim t_k C(\bI)= \dim t_k C_A(\bI) -s(j_1,j_2,k)- \sum_{i\in I_{k-1}'}  (s(j_1,i,k)+s(j_2,i,k)).
\]
\end{proof}

\subsection{Symmetric Dellac configurations}

Recall the set $SDC_N$ of symmetric Dellac configurations (these are the  elements of $S\in DC_N$ such that $RS=S$,
where $R$ is the central reflection). For $S \in SDC_{2n}$, let $INV(S)$ be the set of inversions of $S$.
We say that two inversions $(p_1,p_1')$ and $(p_2,p_2')$ are equivalent, if $R(\{p_1,p_1'\}) = (\{p_2,p_2'\})$.

\begin{dfn}
\label{def:sim}
Let $\widetilde{INV}(S)$ be the set of equivalence classes of inversions of $S$ and let
$\widetilde{inv}(S)$ be the number of elements in $\widetilde{INV}(S)$.
\end{dfn}

For example, in Figure \ref{fig:SpDC4}, we depict the 10 elements $S \in SDC_4$, each of which has its inversions represented by segments. Equivalent inversions are drawed in a same color (the integer $\widetilde{inv}(S)$ is then the number of different colors used in $S$).
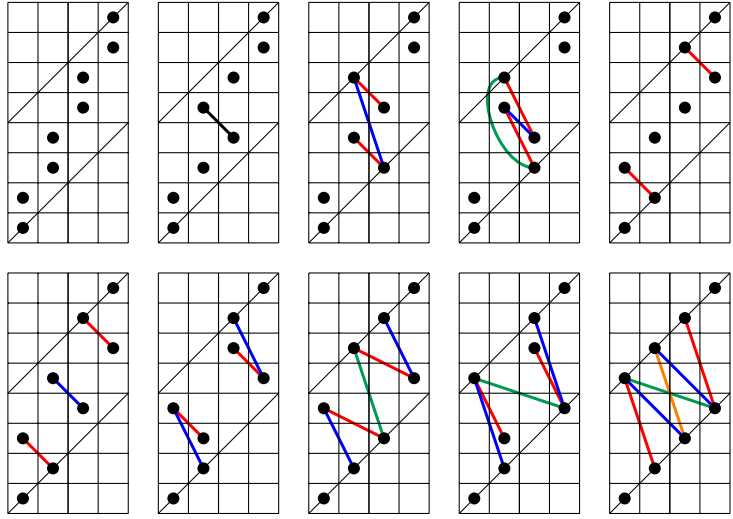
\begin{figure}[!htbp]

\begin{center}

\begin{tikzpicture}[scale=0.4]

\draw (0,0) grid[step=1] (2,8);
\draw (0,0) -- (2,2);
\draw (0,4) -- (2,6);
\fill (0.5,0.5) circle (0.2);
\fill (0.5,1.5) circle (0.2);
\fill (1.5,2.5) circle (0.2);
\fill (1.5,3.5) circle (0.2);
\begin{scope}[shift={(4,8)},rotate=180]
\draw (0,0) grid[step=1] (2,8);
\draw (0,0) -- (2,2);
\draw (0,4) -- (2,6);
\fill (0.5,0.5) circle (0.2);
\fill (0.5,1.5) circle (0.2);
\fill (1.5,2.5) circle (0.2);
\fill (1.5,3.5) circle (0.2);
\end{scope}

\begin{scope}[xshift=5cm]
\draw (0,0) grid[step=1] (2,8);
\draw (0,0) -- (2,2);
\draw (0,4) -- (2,6);
\fill (0.5,0.5) circle (0.2);
\fill (0.5,1.5) circle (0.2);
\fill (1.5,2.5) circle (0.2);
\fill (1.5,4.5) circle (0.2);
\begin{scope}[shift={(4,8)},rotate=180]
\draw (0,0) grid[step=1] (2,8);
\draw (0,0) -- (2,2);
\draw (0,4) -- (2,6);
\fill (0.5,0.5) circle (0.2);
\fill (0.5,1.5) circle (0.2);
\fill (1.5,2.5) circle (0.2);
\fill (1.5,4.5) circle (0.2);

\draw[line width=0.4mm] (1.5,4.5) -- (2.5,3.5);
\end{scope}
\end{scope}

\begin{scope}[xshift=10cm]

\draw[line width=0.4mm,color=red] (1.5,3.5) -- (2.5,2.5);
\draw[line width=0.4mm,color=red] (1.5,5.5) -- (2.5,4.5);
\draw[line width=0.4mm,color=blue] (1.5,5.5) -- (2.5,2.5);

\draw (0,0) grid[step=1] (2,8);
\draw (0,0) -- (2,2);
\draw (0,4) -- (2,6);
\fill (0.5,0.5) circle (0.2);
\fill (0.5,1.5) circle (0.2);
\fill (1.5,3.5) circle (0.2);
\fill (1.5,5.5) circle (0.2);
\begin{scope}[shift={(4,8)},rotate=180]
\draw (0,0) grid[step=1] (2,8);
\draw (0,0) -- (2,2);
\draw (0,4) -- (2,6);
\fill (0.5,0.5) circle (0.2);
\fill (0.5,1.5) circle (0.2);
\fill (1.5,3.5) circle (0.2);
\fill (1.5,5.5) circle (0.2);
\end{scope}
\end{scope}

\begin{scope}[xshift=15cm]
\draw[line width=0.4mm,color=red] (1.5,4.5) -- (2.5,2.5);
\draw[line width=0.4mm,color=red] (1.5,5.5) -- (2.5,3.5);
\draw[line width=0.4mm,color=blue] (1.5,4.5) -- (2.5,3.5);
\draw[line width=0.4mm,color=ForestGreen] (1.5,5.5) to[out=180,in=180] (2.5,2.5);

\draw (0,0) grid[step=1] (2,8);
\draw (0,0) -- (2,2);
\draw (0,4) -- (2,6);
\fill (0.5,0.5) circle (0.2);
\fill (0.5,1.5) circle (0.2);
\fill (1.5,4.5) circle (0.2);
\fill (1.5,5.5) circle (0.2);
\begin{scope}[shift={(4,8)},rotate=180]
\draw (0,0) grid[step=1] (2,8);
\draw (0,0) -- (2,2);
\draw (0,4) -- (2,6);
\fill (0.5,0.5) circle (0.2);
\fill (0.5,1.5) circle (0.2);
\fill (1.5,4.5) circle (0.2);
\fill (1.5,5.5) circle (0.2);
\end{scope}
\end{scope}

\begin{scope}[xshift=20cm]
\draw[line width=0.4mm,color=red] (0.5,2.5) -- (1.5,1.5);
\draw[line width=0.4mm,color=red] (2.5,6.5) -- (3.5,5.5);

\draw (0,0) grid[step=1] (2,8);
\draw (0,0) -- (2,2);
\draw (0,4) -- (2,6);
\fill (0.5,0.5) circle (0.2);
\fill (1.5,1.5) circle (0.2);
\fill (0.5,2.5) circle (0.2);
\fill (1.5,3.5) circle (0.2);
\begin{scope}[shift={(4,8)},rotate=180]
\draw (0,0) grid[step=1] (2,8);
\draw (0,0) -- (2,2);
\draw (0,4) -- (2,6);
\fill (0.5,0.5) circle (0.2);
\fill (1.5,1.5) circle (0.2);
\fill (0.5,2.5) circle (0.2);
\fill (1.5,3.5) circle (0.2);
\end{scope}
\end{scope}

\begin{scope}[shift={(0,-9)}]
\draw[line width=0.4mm,color=red] (0.5,2.5) -- (1.5,1.5);
\draw[line width=0.4mm,color=red] (2.5,6.5) -- (3.5,5.5);
\draw[line width=0.4mm,color=blue] (1.5,4.5) -- (2.5,3.5);

\draw (0,0) grid[step=1] (2,8);
\draw (0,0) -- (2,2);
\draw (0,4) -- (2,6);
\fill (0.5,0.5) circle (0.2);
\fill (1.5,1.5) circle (0.2);
\fill (0.5,2.5) circle (0.2);
\fill (1.5,4.5) circle (0.2);
\begin{scope}[shift={(4,8)},rotate=180]
\draw (0,0) grid[step=1] (2,8);
\draw (0,0) -- (2,2);
\draw (0,4) -- (2,6);
\fill (0.5,0.5) circle (0.2);
\fill (1.5,1.5) circle (0.2);
\fill (0.5,2.5) circle (0.2);
\fill (1.5,4.5) circle (0.2);
\end{scope}
\end{scope}

\begin{scope}[shift={(5,-9)}]
\draw[line width=0.4mm,color=red] (0.5,3.5) -- (1.5,2.5);
\draw[line width=0.4mm,color=red] (2.5,5.5) -- (3.5,4.5);
\draw[line width=0.4mm,color=blue] (0.5,3.5) -- (1.5,1.5);
\draw[line width=0.4mm,color=blue] (2.5,6.5) -- (3.5,4.5);

\draw (0,0) grid[step=1] (2,8);
\draw (0,0) -- (2,2);
\draw (0,4) -- (2,6);
\fill (0.5,0.5) circle (0.2);
\fill (1.5,1.5) circle (0.2);
\fill (1.5,2.5) circle (0.2);
\fill (0.5,3.5) circle (0.2);
\begin{scope}[shift={(4,8)},rotate=180]
\draw (0,0) grid[step=1] (2,8);
\draw (0,0) -- (2,2);
\draw (0,4) -- (2,6);
\fill (0.5,0.5) circle (0.2);
\fill (1.5,1.5) circle (0.2);
\fill (1.5,2.5) circle (0.2);
\fill (0.5,3.5) circle (0.2);
\end{scope}
\end{scope}

\begin{scope}[shift={(10,-9)}]
\draw[line width=0.4mm,color=red] (0.5,3.5) -- (2.5,2.5);
\draw[line width=0.4mm,color=red] (1.5,5.5) -- (3.5,4.5);
\draw[line width=0.4mm,color=blue] (0.5,3.5) -- (1.5,1.5);
\draw[line width=0.4mm,color=blue] (2.5,6.5) -- (3.5,4.5);
\draw[line width=0.4mm,color=ForestGreen] (1.5,5.5) -- (2.5,2.5);

\draw (0,0) grid[step=1] (2,8);
\draw (0,0) -- (2,2);
\draw (0,4) -- (2,6);
\fill (0.5,0.5) circle (0.2);
\fill (1.5,1.5) circle (0.2);
\fill (2.5,2.5) circle (0.2);
\fill (0.5,3.5) circle (0.2);
\begin{scope}[shift={(4,8)},rotate=180]
\draw (0,0) grid[step=1] (2,8);
\draw (0,0) -- (2,2);
\draw (0,4) -- (2,6);
\fill (0.5,0.5) circle (0.2);
\fill (1.5,1.5) circle (0.2);
\fill (2.5,2.5) circle (0.2);
\fill (0.5,3.5) circle (0.2);
\end{scope}
\end{scope}

\begin{scope}[shift={(15,-9)}]
\draw[line width=0.4mm,color=red] (0.5,4.5) -- (1.5,2.5);
\draw[line width=0.4mm,color=red] (2.5,5.5) -- (3.5,3.5);
\draw[line width=0.4mm,color=blue] (0.5,4.5) -- (1.5,1.5);
\draw[line width=0.4mm,color=blue] (2.5,6.5) -- (3.5,3.5);
\draw[line width=0.4mm,color=ForestGreen] (0.5,4.5) -- (3.5,3.5);

\draw (0,0) grid[step=1] (2,8);
\draw (0,0) -- (2,2);
\draw (0,4) -- (2,6);
\fill (0.5,0.5) circle (0.2);
\fill (1.5,1.5) circle (0.2);
\fill (1.5,2.5) circle (0.2);
\fill (0.5,4.5) circle (0.2);
\begin{scope}[shift={(4,8)},rotate=180]
\draw (0,0) grid[step=1] (2,8);
\draw (0,0) -- (2,2);
\draw (0,4) -- (2,6);
\fill (0.5,0.5) circle (0.2);
\fill (1.5,1.5) circle (0.2);
\fill (1.5,2.5) circle (0.2);
\fill (0.5,4.5) circle (0.2);
\end{scope}
\end{scope}

\begin{scope}[shift={(20,-9)}]
\draw[line width=0.4mm,color=red] (0.5,4.5) -- (1.5,1.5);
\draw[line width=0.4mm,color=red] (2.5,6.5) -- (3.5,3.5);
\draw[line width=0.4mm,color=blue] (0.5,4.5) -- (2.5,2.5);
\draw[line width=0.4mm,color=blue] (1.5,5.5) -- (3.5,3.5);
\draw[line width=0.4mm,color=ForestGreen] (0.5,4.5) -- (3.5,3.5);
\draw[line width=0.4mm,color=orange] (1.5,5.5) -- (2.5,2.5);

\draw (0,0) grid[step=1] (2,8);
\draw (0,0) -- (2,2);
\draw (0,4) -- (2,6);
\fill (0.5,0.5) circle (0.2);
\fill (1.5,1.5) circle (0.2);
\fill (2.5,2.5) circle (0.2);
\fill (0.5,4.5) circle (0.2);
\begin{scope}[shift={(4,8)},rotate=180]
\draw (0,0) grid[step=1] (2,8);
\draw (0,0) -- (2,2);
\draw (0,4) -- (2,6);
\fill (0.5,0.5) circle (0.2);
\fill (1.5,1.5) circle (0.2);
\fill (2.5,2.5) circle (0.2);
\fill (0.5,4.5) circle (0.2);
\end{scope}
\end{scope}

\end{tikzpicture}

\end{center}

\caption{The $10$ elements of $Sp DC_4$.}
\label{fig:SpDC4}

\end{figure}

The polynomial $\sum_{S \in SDC_4} q^{\widetilde{inv}(S)} = 1 + 2q + 3q^2 + 3q^3 + q^4$ turns out to be the Poincar\'e polynomial of 
$SpF_{4}^a$. To prove the general statement (see Theorem \ref{thm:dimDellac}), let us prepare some notation.
 
For all $k \in \{1,\hdots,n\}$ and $l \in \{k+1,\hdots,2n,1,\hdots,k\}$, we define $l_k$ as $l$ if $l \in \{k+1,\hdots,2n\}$, 
and as $2n+l$ otherwise.

For a point $p = (j,i)$ of a symplectic Dellac configuration $S \in SDC_{2n}$, the point $(2n+1-j,4n+1-i)$ of $S$ (symmetrical to $p$ with respect to the center of $S$) is denoted by $p_{sym}$. Note that if $p = (j,l_j)$, then $p_{sym} = (2n+1-j,(2n+1-l)_j)$.

\begin{thm}
\label{thm:dimDellac}
Consider a collection ${\bf I}=(I_1,\dots,I_n) \in \mathcal{I}_{2n}$, the corresponding cell $C(\bI)$, and the corresponding 
symmetric  Dellac configuration $S(\bI)$. We have
$$\dim C(\bI) = \widetilde{inv}(S(\bI)).$$
\end{thm}

\begin{proof}
Recall~\cite{F1} that the bijection $\bI \in \mathcal{I}_{2n} \mapsto S(\bI) \in SDC_{2n}$ consists of:
\begin{enumerate}[label=(\alph*)]
\item drawing an empty tableau $S$ made of $2n$ columns and $4n$ rows;
\item printing the parallelogram $\mathcal{P}(\bI)$ in such a way that the disk $[1:2]$ of $\mathcal{P}(\bI)$ is located in the box $(1,2_1) = (1,2)$ of $S$;
\item erasing all the disks, except the initiating ones (which we paint in blue);
\item for all $i \in \{1,2,\hdots,n\}$, if the $i$-th row (from bottom to top) of $S$ is still empty, then printing a red disk in the box $(i,i)$;
\item finally, applying the central reflection with respect to the center of $S$ to the points of its $n$ first columns.
\end{enumerate}
For the example $\bI = \bI_0 \in \mathcal{I}_8$ of Figure \ref{fig:PI0}, we obtain the following :
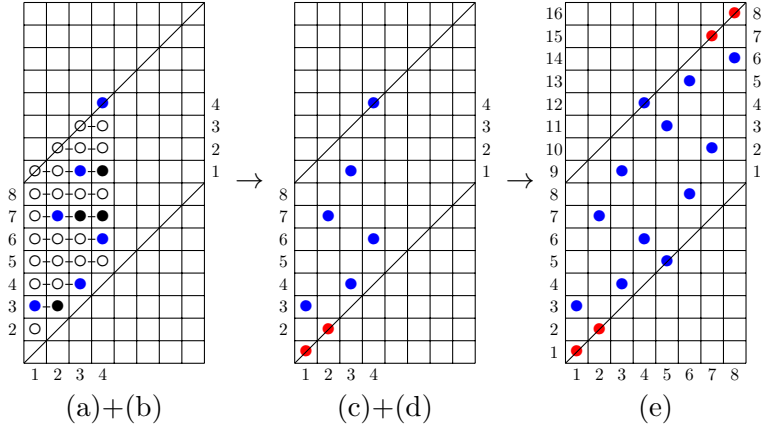
\begin{center}
\begin{tikzpicture}[scale=0.3]

\draw (0,0) grid[step=1] (8,16);

\begin{scope}[xshift=-0.5cm,yshift=-0.5cm]

\draw (1,2) node[scale=1]{$\circ$};
\draw [color=blue] (1,3) node[scale=1]{$\bullet$};
\draw (1,4) node[scale=1]{$\circ$};
\draw (1,5) node[scale=1]{$\circ$};
\draw (1,6) node[scale=1]{$\circ$};
\draw (1,7) node[scale=1]{$\circ$};
\draw (1,8) node[scale=1]{$\circ$};
\draw (1,9) node[scale=1]{$\circ$};

\draw (2,3) node[scale=1]{$\bullet$};
\draw (2,4) node[scale=1]{$\circ$};
\draw (2,5) node[scale=1]{$\circ$};
\draw (2,6) node[scale=1]{$\circ$};
\draw [color=blue] (2,7) node[scale=1]{$\bullet$};
\draw (2,8) node[scale=1]{$\circ$};
\draw (2,9) node[scale=1]{$\circ$};
\draw (2,10) node[scale=1]{$\circ$};

\draw [color=blue] (3,4) node[scale=1]{$\bullet$};
\draw (3,5) node[scale=1]{$\circ$};
\draw (3,6) node[scale=1]{$\circ$};
\draw (3,7) node[scale=1]{$\bullet$};
\draw (3,8) node[scale=1]{$\circ$};
\draw [color=blue] (3,9) node[scale=1]{$\bullet$};
\draw (3,10) node[scale=1]{$\circ$};
\draw (3,11) node[scale=1]{$\circ$};

\draw (4,5) node[scale=1]{$\circ$};
\draw [color=blue] (4,6) node[scale=1]{$\bullet$};
\draw (4,7) node[scale=1]{$\bullet$};
\draw (4,8) node[scale=1]{$\circ$};
\draw (4,9) node[scale=1]{$\bullet$};
\draw (4,10) node[scale=1]{$\circ$};
\draw (4,11) node[scale=1]{$\circ$};
\draw [color=blue] (4,12) node[scale=1]{$\bullet$};

\draw (1,3) [shorten <= 0.1cm, shorten >= 0.1cm] -- (2,3);
\draw (1,4) [shorten <= 0.1cm, shorten >= 0.1cm] -- (2,4);
\draw (2,4) [shorten <= 0.1cm, shorten >= 0.1cm] -- (3,4);
\draw (1,5) [shorten <= 0.1cm, shorten >= 0.1cm] -- (2,5);
\draw (2,5) [shorten <= 0.1cm, shorten >= 0.1cm] -- (3,5);
\draw (3,5) [shorten <= 0.1cm, shorten >= 0.1cm] -- (4,5);
\draw (1,6) [shorten <= 0.1cm, shorten >= 0.1cm] -- (2,6);
\draw (2,6) [shorten <= 0.1cm, shorten >= 0.1cm] -- (3,6);
\draw (3,6) [shorten <= 0.1cm, shorten >= 0.1cm] -- (4,6);
\draw (1,7) [shorten <= 0.1cm, shorten >= 0.1cm] -- (2,7);
\draw (2,7) [shorten <= 0.1cm, shorten >= 0.1cm] -- (3,7);
\draw (3,7) [shorten <= 0.1cm, shorten >= 0.1cm] -- (4,7);
\draw (1,8) [shorten <= 0.1cm, shorten >= 0.1cm] -- (2,8);
\draw (2,8) [shorten <= 0.1cm, shorten >= 0.1cm] -- (3,8);
\draw (3,8) [shorten <= 0.1cm, shorten >= 0.1cm] -- (4,8);
\draw (1,9) [shorten <= 0.1cm, shorten >= 0.1cm] -- (2,9);
\draw (2,9) [shorten <= 0.1cm, shorten >= 0.1cm] -- (3,9);
\draw (3,9) [shorten <= 0.1cm, shorten >= 0.1cm] -- (4,9);
\draw (2,10) [shorten <= 0.1cm, shorten >= 0.1cm] -- (3,10);
\draw (3,10) [shorten <= 0.1cm, shorten >= 0.1cm] -- (4,10);
\draw (3,11) [shorten <= 0.1cm, shorten >= 0.1cm] -- (4,11);

\end{scope}

\draw (-0.5,1.5) node[scale=0.6]{$2$};
\draw (-0.5,2.5) node[scale=0.6]{$3$};
\draw (-0.5,3.5) node[scale=0.6]{$4$};
\draw (-0.5,4.5) node[scale=0.6]{$5$};
\draw (-0.5,5.5) node[scale=0.6]{$6$};
\draw (-0.5,6.5) node[scale=0.6]{$7$};
\draw (-0.5,7.5) node[scale=0.6]{$8$};

\draw (8.5,8.5) node[scale=0.6]{$1$};
\draw (8.5,9.5) node[scale=0.6]{$2$};
\draw (8.5,10.5) node[scale=0.6]{$3$};
\draw (8.5,11.5) node[scale=0.6]{$4$};

\draw (0.5,-0.5) node[scale=0.6]{$1$};
\draw (1.5,-0.5) node[scale=0.6]{$2$};
\draw (2.5,-0.5) node[scale=0.6]{$3$};
\draw (3.5,-0.5) node[scale=0.6]{$4$};

\draw (0,0) -- (8,8);
\draw (0,8) -- (8,16);

\draw (10,8) node[scale=1]{$\rightarrow$};

\draw (4,-2) node[scale=1]{$\text{(a)+(b)}$};

\begin{scope}[xshift=12cm]

\draw (0,0) grid[step=1] (8,16);

\begin{scope}[xshift=-0.5cm,yshift=-0.5cm]

\draw [color=blue] (1,3) node[scale=1]{$\bullet$};
\draw [color=blue] (2,7) node[scale=1]{$\bullet$};
\draw [color=blue] (3,4) node[scale=1]{$\bullet$};
\draw [color=blue] (3,9) node[scale=1]{$\bullet$};
\draw [color=blue] (4,6) node[scale=1]{$\bullet$};
\draw [color=blue] (4,12) node[scale=1]{$\bullet$};

\draw [color=red] (1,1) node[scale=1]{$\bullet$};
\draw [color=red] (2,2) node[scale=1]{$\bullet$};

\end{scope}

\draw (-0.5,1.5) node[scale=0.6]{$2$};
\draw (-0.5,2.5) node[scale=0.6]{$3$};
\draw (-0.5,3.5) node[scale=0.6]{$4$};
\draw (-0.5,4.5) node[scale=0.6]{$5$};
\draw (-0.5,5.5) node[scale=0.6]{$6$};
\draw (-0.5,6.5) node[scale=0.6]{$7$};
\draw (-0.5,7.5) node[scale=0.6]{$8$};

\draw (8.5,8.5) node[scale=0.6]{$1$};
\draw (8.5,9.5) node[scale=0.6]{$2$};
\draw (8.5,10.5) node[scale=0.6]{$3$};
\draw (8.5,11.5) node[scale=0.6]{$4$};

\draw (0.5,-0.5) node[scale=0.6]{$1$};
\draw (1.5,-0.5) node[scale=0.6]{$2$};
\draw (2.5,-0.5) node[scale=0.6]{$3$};
\draw (3.5,-0.5) node[scale=0.6]{$4$};

\draw (0,0) -- (8,8);
\draw (0,8) -- (8,16);

\draw (10,8) node[scale=1]{$\rightarrow$};

\draw (4,-2) node[scale=1]{$\text{(c)+(d)}$};

\end{scope}

\begin{scope}[xshift=24cm]

\draw (0,0) grid[step=1] (8,16);

\begin{scope}[xshift=-0.5cm,yshift=-0.5cm]

\draw [color=blue] (1,3) node[scale=1]{$\bullet$};
\draw [color=blue] (2,7) node[scale=1]{$\bullet$};
\draw [color=blue] (3,4) node[scale=1]{$\bullet$};
\draw [color=blue] (3,9) node[scale=1]{$\bullet$};
\draw [color=blue] (4,6) node[scale=1]{$\bullet$};
\draw [color=blue] (4,12) node[scale=1]{$\bullet$};

\draw [color=blue] (8,14) node[scale=1]{$\bullet$};
\draw [color=blue] (7,10) node[scale=1]{$\bullet$};
\draw [color=blue] (6,13) node[scale=1]{$\bullet$};
\draw [color=blue] (6,8) node[scale=1]{$\bullet$};
\draw [color=blue] (5,11) node[scale=1]{$\bullet$};
\draw [color=blue] (5,5) node[scale=1]{$\bullet$};

\draw [color=red] (1,1) node[scale=1]{$\bullet$};
\draw [color=red] (2,2) node[scale=1]{$\bullet$};
\draw [color=red] (8,16) node[scale=1]{$\bullet$};
\draw [color=red] (7,15) node[scale=1]{$\bullet$};

\end{scope}

\draw (-0.5,0.5) node[scale=0.6]{$1$};
\draw (-0.5,1.5) node[scale=0.6]{$2$};
\draw (-0.5,2.5) node[scale=0.6]{$3$};
\draw (-0.5,3.5) node[scale=0.6]{$4$};
\draw (-0.5,4.5) node[scale=0.6]{$5$};
\draw (-0.5,5.5) node[scale=0.6]{$6$};
\draw (-0.5,6.5) node[scale=0.6]{$7$};
\draw (-0.5,7.5) node[scale=0.6]{$8$};

\draw (8.5,8.5) node[scale=0.6]{$1$};
\draw (8.5,9.5) node[scale=0.6]{$2$};
\draw (8.5,10.5) node[scale=0.6]{$3$};
\draw (8.5,11.5) node[scale=0.6]{$4$};
\draw (8.5,12.5) node[scale=0.6]{$5$};
\draw (8.5,13.5) node[scale=0.6]{$6$};
\draw (8.5,14.5) node[scale=0.6]{$7$};
\draw (8.5,15.5) node[scale=0.6]{$8$};

\draw (-0.5,8.5) node[scale=0.6]{$9$};
\draw (-0.5,9.5) node[scale=0.6]{$10$};
\draw (-0.5,10.5) node[scale=0.6]{$11$};
\draw (-0.5,11.5) node[scale=0.6]{$12$};
\draw (-0.5,12.5) node[scale=0.6]{$13$};
\draw (-0.5,13.5) node[scale=0.6]{$14$};
\draw (-0.5,14.5) node[scale=0.6]{$15$};
\draw (-0.5,15.5) node[scale=0.6]{$16$};

\draw (0.5,-0.5) node[scale=0.6]{$1$};
\draw (1.5,-0.5) node[scale=0.6]{$2$};
\draw (2.5,-0.5) node[scale=0.6]{$3$};
\draw (3.5,-0.5) node[scale=0.6]{$4$};
\draw (4.5,-0.5) node[scale=0.6]{$5$};
\draw (5.5,-0.5) node[scale=0.6]{$6$};
\draw (6.5,-0.5) node[scale=0.6]{$7$};
\draw (7.5,-0.5) node[scale=0.6]{$8$};

\draw (0,0) -- (8,8);
\draw (0,8) -- (8,16);

\draw (4,-2) node[scale=1]{$\text{(e)}$};

\end{scope}

\end{tikzpicture}
\end{center}
In the computation of $d_1 \leq d_2 \leq \hdots \leq d_n = \dim(C(\bI))$, it is then straightforward that $d_1$ is the 
number of inversions whose first element is the upper point $p_1$ of the first column of $S(\bI)$, which is also the 
number of elements of $\widetilde{INV}(S(\bI))$ containing an inversion whose first point is $p_1$.

Suppose now that for some $k \geq 2$, the integer $d_{k-1}$ is the number of elements of $\widetilde{INV}(S(\bI))$ 
containing an inversion whose first point is located in one of the first $k-1$ columns of $S(\bI)$. Let $[k:l] \in ID_k(\bI)$. 
By definition of $S(\bI)$, the box $(k,l_k)$ of $S(\bI)$ contains a point $p$. Here again, it is straightforward that $w(D)$ 
is the number of inversions of $S(\bI)$ whose first point is $p$. Afterwards, let $D' = [k:l'] \in B(D)$ for some $l'<_k l$, 
and let $k' \leq k$ such that $(k',l') \in ID_{k'}(\bI)$, hence the box $(k',l'_{k'})$ of $S(\bI)$ contains a point $p'$. 
Consider the points $p_{sym} = (2n+1-k,(2n+1-l)_k)$ and $p'_{sym} = (2n+1-k',(2n+1-l')_{k'}) = (2n+1-k',(2n+1-l')_{k})$ of $S(\bI)$. 
The following properties are equivalent :
\begin{enumerate}
\item $(p,p'_{sym}) \in INV(S(\bI))$;
\item $(p',p_{sym}) \in INV(S(\bI))$;
\item $l_k > (2n+1-l')_k$;
\item $s(l,l',k) = 1$. 
\end{enumerate}
The equivalences $(i) \Leftrightarrow (ii) \Leftrightarrow (iii)$ are obvious. Now, since the dots of a Dellac configuration are located above the line $ y = x$, the existence of the point $p'_{sym} = (2n+1-k',(2n+1-l')_k)$ implies
$$2n-k \leq 2n-k' < (2n+1-l')_k,$$
so $(iii)$ is equivalent to 
$$2n-k <_k 2n+1-l' <_k l,$$
which is equivalent to $(iv)$.
In Figure \ref{fig:description}, we illustrate a situation with 3 points $p,p',p''$ where $p$ and $p'$ have the properties 
$(i),...,(iv)$, as opposed to $p$ and $p''$. The inversions $(p,p'_{sym})$ and $(p',p_{sym})$ having the same class in $INV(S(\bI))$ 
are expressed by drawing them in red.
\begin{figure}
\begin{tikzpicture}[scale=0.58]

\draw (-0.5,1.5) node[scale=0.8]{$2$};
\draw (-0.5,4.5) node[scale=0.8]{$k$};
\draw (-1,5.5) node[scale=0.8]{$k+1$};
\draw (-1,11.5) node[scale=0.8]{$2n-k$};
\draw (-0.5,15.5) node[scale=0.8]{$2n$};

\draw (4.5,18.5) node[scale=1]{$p$};
\draw (16.5,18.5) node[scale=0.8]{$l$};
\draw (3.5,14.5) node[scale=1]{$p'$};
\draw (-0.5,14.5) node[scale=0.8]{$l'$};
\draw (2.5,8.5) node[scale=1]{$p''$};
\draw (-0.5,8.5) node[scale=0.8]{$l''$};

\draw (11.55,13.5) node[scale=0.75]{$p_{sym}$};
\draw (-1.5,13.5) node[scale=0.8]{$2n+1-l$};
\draw (-1.5,13.5) node[scale=0.8]{$2n+1-l$};
\draw (12.65,17.5) node[scale=0.9]{$p'_{sym}$};
\draw (17.5,17.5) node[scale=0.8]{$2n+1-l'$};
\draw (13.4,23.5) node[scale=0.9]{$p''_{sym}$};
\draw (17.5,23.5) node[scale=0.8]{$2n+1-l''$};

\draw [line width=0.4mm,color=black,opacity=0.5] (4.5,18.5) [shorten <= 0.2cm, shorten >= 0.4cm]  -- (11.5,13.5);
\draw [line width=0.4mm,color=red] (3.5,14.5) [shorten <= 0.2cm, shorten >= 0.4cm]  -- (11.5,13.5);
\draw [line width=0.4mm,color=red] (4.5,18.5) [shorten <= 0.2cm, shorten >= 0.4cm]  -- (12.5,17.5);

\draw (16.5,16.5) node[scale=0.8]{$1$};
\draw (16.5,20.5) node[scale=0.8]{$k$};
\draw (17,30.5) node[scale=0.8]{$2n-1$};

\draw (0.5,-0.5) node[scale=0.8]{$1$};
\draw (1.5,-0.5) node[scale=0.8]{$2$};
\draw (4.5,-0.5) node[scale=0.8]{$k$};
\draw (7.5,-0.5) node[scale=0.8]{$n$};
\draw (8.5,-0.5) node[scale=0.8]{$n+1$};
\draw (11.5,-0.5) node[scale=0.8]{$2n+1-k$};
\draw (15.5,-0.5) node[scale=0.8]{$2n$};

\draw (0,21) -- (16,21);
\draw (0,20) -- (16,20);
\draw (0,11) -- (16,11);
\draw (0,12) -- (16,12);
\draw (0,6) -- (16,6);
\draw (0,5) -- (16,5);
\draw (0,4) -- (16,4);
\draw (0,0) -- (16,16);
\draw (0,16) -- (16,32);
\draw (0,0) rectangle (16,32);
\draw (8,0) -- (8,32);
\draw (7,0) -- (7,32);
\draw (9,0) -- (9,32);
\draw (4,0) -- (4,32);
\draw (5,0) -- (5,32);
\draw (11,0) -- (11,32);
\draw (12,0) -- (12,32);
\draw (0,1) -- (16,1);
\draw (0,2) -- (16,2);
\draw (0,15) -- (16,15);
\draw (0,16) -- (16,16);
\draw (0,17) -- (16,17);
\draw (0,32) -- (16,32);
\draw (0,31) -- (16,31);
\draw (1,0) -- (1,32);
\draw (2,0) -- (2,32);
\draw (15,0) -- (15,32);
\draw (0,30) -- (16,30);
\draw (14,0) -- (14,32);

\end{tikzpicture}
\caption{Symmetric Dellac configuration $S(\bI)$.}
\label{fig:description}
\end{figure}
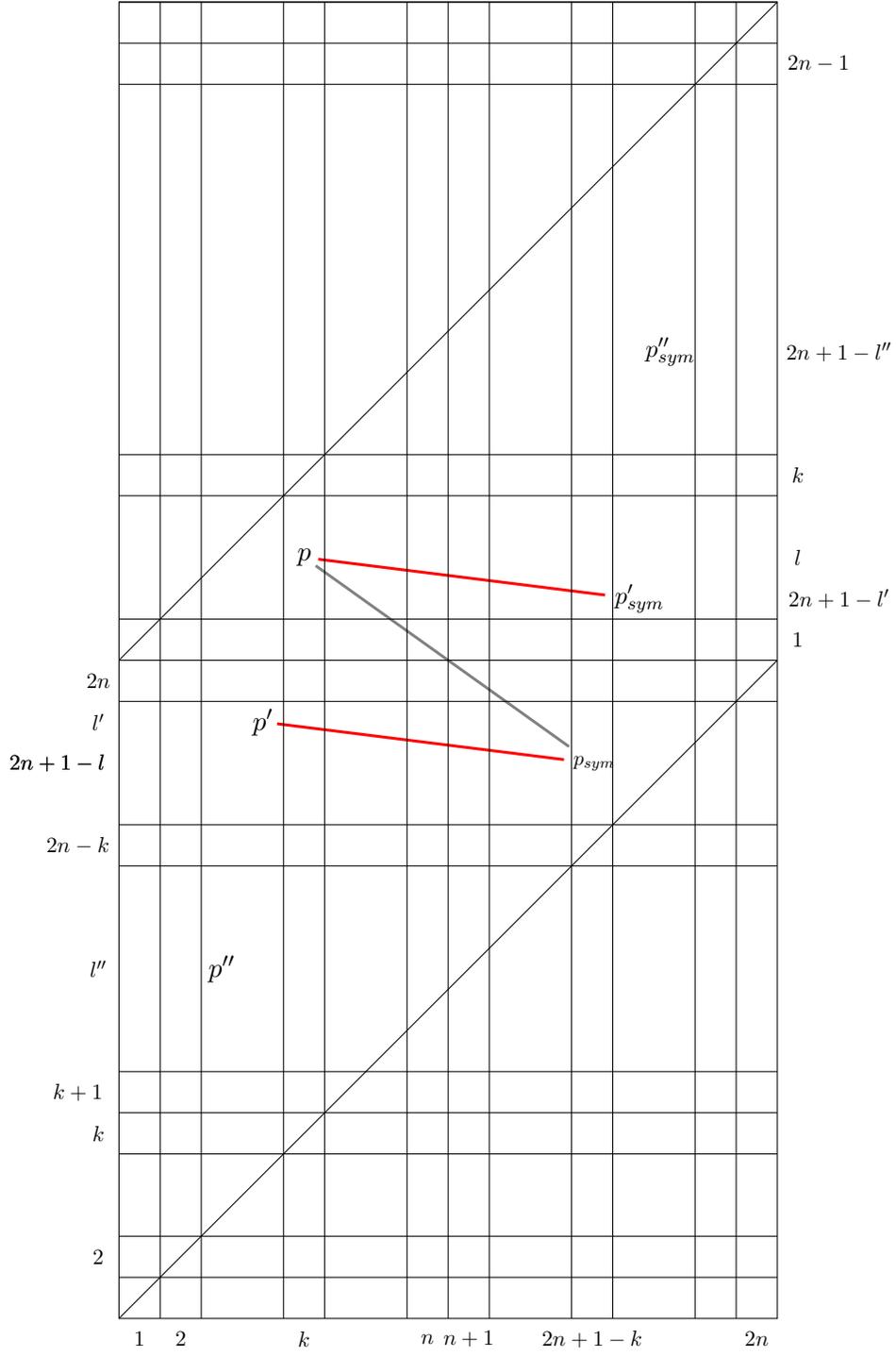

In general, the equivalence $(i) \Leftrightarrow (iv)$ implies that the integer $w(D) - \sum_{D' = [k:l'] \in B(D)} s(l,l',k)$ is the number of elements of $\widetilde{INV}(S(\bI))$ that contain an inversion of $S(\bI)$ whose first point is $p$, but no inversion whose first point is located in one of the $k-1$ first columns of $S(\bI)$. By hypothesis of induction, and in view of 
formula (\ref{eq:inductiondk}), we obtain that $d_k$ is the number of elements of $\widetilde{INV}(S(\bI))$ containing an inversion whose first point is located in one of the first $k$ columns of $S(\bI)$. By induction, we have in particular $\dim(C(\bI)) = d_n =  |\widetilde{INV}(S(\bI))|$.
\end{proof}

\section{Degenerate flags: odd symplectic case}
\label{OSC}
Let us consider a $(2n+1)$-dimensional vector space $E$ with a fixed basis $(e_1,\dots,e_{2n+1})$ equipped with the
skew-symmetric bilinear form of maximal possible rank with non-trivial pairings $(e_i,e_{2n+2-i})=1$ for $i=1,\dots,n$. 
In particular, $e_{n+1}$ spans the kernel of this form. 

We define the variety $SpF^a_{2n+1}$ as the variety of collections of subspaces $(V_k)_{k=1}^n$, $V_k\subset E$ such that
$\dim V_k=k$, $pr_{k+1} V_k\subset V_{k+1}$ and $V_n$ is isotropic. We note that $V_n$ is isotropic if and only if
$(V_n,pr_{n+1}V_n)=0$. The coordinate flags in $SpF^a_{2n+1}$ 
are labeled by collections $(I_k)_{k=1}^n$
of subsets of $\{1,\dots,2n+1\}$ such that $I_k\subset I_{k+1}\cup \{k+1\}$ and $a\in I_n$ implies $2n+2-a\notin I_n$ for $1\le a\le n$.
We denote the set of such collections by ${\mathcal I}_{2n+1}$.

\begin{lem}\label{cf-sdc}
The number of the coordinate flags in $SpF^a_{2n+1}$ is equal to the number of odd symmetric Dellac configurations.
\end{lem}
\begin{proof}
Recall (see \cite{F1} and the proof of Theorem \ref{thm:dimDellac}) the bijection $S$ between 
coordinate flags in $F^a_{2n+1}$ and Dellac configurations.
One easily sees that the restriction of this map to the set of coordinate flags of $SpF^a_{2n+1}$ is a bijection 
to the set of odd symmetric Dellac configurations.   
\end{proof}

Let us denote the number of elements of $SDC_{2n+1}$ by $l_n$, $n\ge 0$. 
These numbers start with $1,3,21,267,5349$.

\begin{rem}
According to \cite{BF} the numbers $l_n$ show up in the paper \cite{RZ}.
\end{rem}

For $1\le k\le n$ we use the order $<_k$ on the set $\{1,\dots,2n+1\}$:
\[
k+1<_k k+2<_k\dots <_k 2n+1<_k 1<_k\dots <_k k. 
\]
We now construct cellular decomposition of $SpF^a_{2n+1}$. Given an element $(I_k)_{k=1}^n\in {\mathcal I}_{2n+1}$ we define the cell
$C(\bI)$ as the set of collections $(V_k)_{k=1}^n\in SpF^a_{2n+1}$ such that $V_k$ has a basis of the form   
\[
e_a+\sum_{b<_k a} c_{a,b} e_b,\ a\in I_k.
\]
Clearly, $C(\bI)$ is non empty if and only if $\bI$ defines a coordinate flag of $SpF^a_{2n+1}$. 

The proof of the following lemma is very similar to the proof of Lemma \ref{lem:ac} and 
Theorem \ref{thm:algorithm}.

\begin{lem}
$C(\bI)$ is an affine cell, whose complex dimension is computed by the following rule: we first consider all initiating
elements in all $I_k$ and compute the number of empty slots (not belonging to $I_k$) below these initiating elements 
(in the order $<_k$). After that we subtract the number of pairs $a,b\in I_n$ such that $2n+2-b>_n a>_n b$. 
\end{lem}

Recall the bijection $S$ between ${\mathcal I}_{2n+1}$ and the symmetric Dellac configurations of size $(2n+1)\times (4n+2)$. 
Recall the statistics ${\widetilde{inv}}(D)$ on the set of symmetric Dellac configurations
(counting the number of inversions modulo the central symmetry). 

The proof of the following proposition is very similar to the proof of Theorem \ref{thm:dimDellac}.  
\begin{prop}
$\dim C(\bI)={\widetilde{inv}} S(\bI)$.
\end{prop}

\begin{rem}
Given a point $(V_k)_{k=1}^n\in SpF^a_{2n+1}$ there is no canonical way to complete it to an element of $F^a_{2n+1}$
(i.e. to complete it with $V_{n+1},\dots,V_{2n}$). The reason is that the form is degenerate and thus the orthogonal complement
to, say, $V_n$ can be of dimension $n+1$ or $n+2$. 
We note however that a symmetric Dellac configuration $D$ of size $(2n+1)\times (4n+2)$ does define an element of
$F^a_{2n+1}$ (a coordinate flag). The meaning of the last $n$ components is not clear from symplectic point of view,
but is very natural from the orthogonal story below. 
\end{rem}

\begin{thm}
The variety $SpF^a_{2n+1}$ is irreducible of dimension $n(n+1)$. It carries an action of abelian unipotent group 
with an open dense orbit. 
\end{thm}
\begin{proof}
We construct the resolution of singularities of $SpF^a_{2n+1}$ in the spirit of \cite{FFi,FFL} (see also 
section \ref{resofsing}).
Let $SpR_{2n+1}$ be the variety of collections $V_{i,j}$, $1\le i\le n$, $i+j\le 2n+1$ of  subspaces
of $E$ satisfying the following conditions:
\begin{itemize}
\item $\dim V_{i,j}=i$,\ $V_{i,j}\subset {\mathrm{span}} (e_1,\dots,e_i, e_{j+1},e_{j+1},\dots,e_{2n+1})$,
\item $V_{i,j}\subset V_{i+1,j}$,\ $pr_{j+1} V_{i,j}\subset V_{i,j+1}$,
\item $V_{i,2n+1-i}$ is isotropic.
\end{itemize}   
We note that $V_{i,2n+1-i}$ is a subspace of ${\mathrm{span}} (e_1,\dots,e_i, e_{2n+2-i},\dots,e_{2n+1})$ and 
the restriction of our form to this space is non-degenerate. Hence each $V_{i,2n+2-i}$ is an element of the corresponding
Lagrangian Grassmannian.

One easily sees that $SpR_{2n+1}$ is a tower of $\bP^1$ fibrations of dimension $n(n+1)$. Now let $SpR^0_{2n+1}\subset SpR_{2n+1}$
be the open subvariety defined by the condition that the Pl\"ucker coordinates $X_{1,\dots,i}$ of $V_{i,i}$ do not vanish 
for all $i=1,\dots,n$. Then the natural forgetful map
$$\pi: SpR_{2n+1}\to SpF^a_{2n+1}, \ \pi (V_{i,j})_{i,j}\mapsto (V_{i,i})_{i=1}^n$$
is surjective and restricts to an isomorphism on $SpR^0_{2n+1}$.

Finally, let us construct the action of an additive unipotent group on $SpF^a_{2n+1}$. 
An element of this group is a collection $(g_1,\dots,g_n)$ of automorphisms of the space $E$.
The operators $g_k$ should satisfy the following properties:
\begin{itemize}
\item the matrix elements $(g_k)_{i,j}$ are zero unless $j\ge_k 1$, $i<_k 1$,
\item $pr_{k+1}g_k=g_{k+1}pr_{k+1}$ for $k=1,\dots,n-1$;
\item $g_n{\mathrm{span}}(e_1,\dots,e_n)$ is isotropic. 
\end{itemize}
Analogously to the even symplectic case \cite{FFL} one sees that this group is abelian, unipotent and 
the orbit through the point $({\mathrm{span}}(e_1,\dots,e_k))_{k=1}^n$ is an open dense subset of 
$SpF^a_{2n+1}$.
\end{proof}

\begin{rem}
It is tempting to conjecture that $SpF^a_{2n+1}$ is isomorphic to the orbit closure in the projectivized PBW 
degenerate representation of the odd symplectic group.
\end{rem}

\section{Degenerate flags: orthogonal case}
\label{OC}
In this section we introduce and study the orthogonal couterpart of the symplectic story.
We define the corresponding algebraic varieties and study the link between their topology and 
the combinatorics of the symmetric Dellac configurations. We do both odd and even 
orthogonal cases simultaneously providing some details separately for types $B$ and $D$ if necessary.  

Let us consider an $N$-dimensional vector space $E$ with basis $e_1,\dots,e_N$ equipped with symmetric
non-degenerate bilinear form $(e_i,e_j)=\delta_{i+j,N+1}$. 
We fix $n=\floor{N/2}$.

\begin{dfn}
The orthogonal version $SOF^a_N$ of the type $A$ degenerate flag variety is a subvariety of 
$F^a_N$ such that $V_k^\perp=V_{N-k}$. In other words, $SOF^a_N$ consists of collections  $(V_k)_{k=1}^n$
of subspaces of $E$ such that $\dim V_k=k$, $pr_{k+1}V_k\subset V_{k+1}$ and 
\[
(V_n,V_n)=0 \text{ for even } N \text{ and } (V_n,pr_{n+1} V_n)=0 \text{ for odd } N. 
\]
\end{dfn}

\begin{rem}
The following example shows that in contrast to the type $A$ and type $C$ cases the varieties we get
in the orthogonal case are no longer irreducible. In particular, our $SOF^a_N$ is NOT isomorphic 
to the PBW degeneration of $SO_N/B$, where $B$ is a Borel subgroup. However, an irreducible component
of $SO_N/B$ might be an appropriate candidate.
\end{rem}

\begin{example}
Let $\dim E=3$. Then $SOF^a_3$ consists of pairs $(V_1,V_2=V_1^\perp)$ of subspaces of $E$ such that $pr_2 V_1$ is isotropic.
If $ae_1+be_2+ce_3$ is the spanning vector of $V_1$, then one needs $ac=0$. We thus conclude that $SOF^a_3$
is reducible with two irreducible components isomorphic to $\bP^1$ glued at a point. 

Let $\dim E=4$. Then $SOF^a_4$ consists of triples  $(V_1,V_2,V_1^\perp)$ of subspaces of $E$ such that $V_2$ is isotropic
and $pr_2V_1\subset V_2$.
Recall that the isotropic Grassmannian of $2$-dimensional subspaces of a four dimensional space 
has two irreducible components both of dimension one. One can see that $SOF^a_4$ is of dimension two and has four
irreducible components.  
\end{example}  

We start with adopting the resolution of singularities construction to this case.
\begin{dfn}\label{SOR}
$SOR_N$ ($N=2n$ or $N=2n+1$) is the variety of collections $(V_{i,j})$, $1\le i\le j\le n$, $i+j\le N$ 
such that 
\begin{itemize}
\item $\dim V_{i,j}=i$,\ $V_{i,j}\subset {\mathrm{span}} (e_1,\dots,e_i, e_{j+1},e_{j+1},\dots,e_N)$,
\item $V_{i,j}\subset V_{i+1,j}$,\ $pr_{j+1} V_{i,j}\subset V_{i,j+1}$,
\item $V_{i,N-i}$ is isotropic.
\end{itemize}   
\end{dfn}

In particular, each point of $SOR_N$ is determined by $N^2/4$ subspaces for even $N$ and by 
$(N^2-1)/4$ subspaces for odd $N$ (i.e. has $\lfloor N^2/4\rfloor$ components). 
We note that $V_{i,N-i}$ is a subspace of ${\mathrm{span}} (e_1,\dots,e_i, e_{N+1-i},\dots,e_N)$ and 
and hence an element of the orthogonal Grassmannian $OGr(i,2i)$. 

\begin{thm}
$SOR_N$ is smooth disjoint union of $2^n$ components. Each component is a tower of $\bP^1$ fibrations of dimension
$(N-1)^2/4$ for odd $N$ and $N(N-2)/4$ for even $N$.
The map $SOR_N\to SOF^a_N$ forgetting off-diagonal elements of a collection $(V_{i,j})$ is surjective and a birational isomorphism. 
\end{thm}
\begin{proof}
A point of $SOR_N$ is determined by the collection $(V_{i,j})$ with $1\le i\le j\le n$, $i+j\le N$.  
We approximate $SOR_N$ by the chain $R_k$, $k=1,\dots, \lfloor N^2/4\rfloor$ of varieties such that 
$R_0={\rm pt}\sqcup {\rm pt}$, $R_{\lfloor N^2/4\rfloor}=SOR_N$ and each $R_k$ is either fibered over $R_{k-1}$ with a fiber 
being $\bP^1$ or $R_k$ is isomorphic to a disjoint union of two copies of $R_{k-1}$.

We order all pairs $(i,j)$, $1\le i\le n$, $i+j\le N$ into the sequence $(i_k,j_k)$, 
$k=1,\dots, \lfloor N^2/4\rfloor$ in the following way: $(i_1,j_1)=(1,N-1)$ and $(i_{k+1},j_{k+1})=(i_k+1,j_k)$
if ($i_k+j_k<N$ and $i_k<j_k$) and $(i_{k+1},j_{k+1})=(1,j_k-1)$ otherwise. The last pair in our sequence is 
$(1,1)$. For example, for $N=6$ our ordering looks as follows:
\[
(1,5), (1,4), (2,3), (1,3), (2,3), (3,3), (1,2), (2,2), (1,1).
\]
We define $R_k$ as the collection of $V_{i_l,j_l}$, $l\le k$ such that all the conditions from Definition \ref{SOR} hold. 

Let us compare $R_k$ with $R_{k-1}$. First, let $k=1$. Then $R_1$ consists of $1$-dimensional isotropic subspaces
of a $2$-dimensional vector space spanned by $e_1$ and $e_N$ with $(e_1,e_N)=1$. Clearly, $R_1$ is a disjoint union of two points.
Second, assume $i_k+j_k<N$. Then the same argument as in \cite{FFi,FFL} shows that the natural map $R_{k}\to R_{k-1}$
forgetting the component $V_{i_k,j_k}$ is a $\bP^1$ fibration.
Third, assume $i_k+j_k=N$. We consider the natural forgetful map $R_k\to R_{k-1}$. Assume we are given a point        
$(V_{i_a,j_a})_{a\le k-1}\in R_{k-1}$. Then the preimage of such a point is isomorphic to the variety of   
of isotropic subspaces $V_{i_k,j_k}$ such that 
$${\mathrm{span}}(e_1,\dots,e_{i_k},e_{j_k},\dots,e_N)\supset V_{i_k,j_k}\supset V_{i_k-1,j_k}.$$ 
Since $\dim V_{i_k-1,j_k}=i_k-1$, the preimage is isomorphic to a disjoint union of two points.  
We note that the disjoint components of $SOR_N$ are determined by the parity of the dimensions 
$\dim V_{a,N-a}\cap {\mathrm{span}}(e_1,\dots,e_a)$. 

The last claim of the theorem is proved in the same way as in the symplectic case.
\end{proof}

\begin{cor}
$\dim SOF^a_N=\dim SO_N/B=\dim\fn_-$.
\end{cor}

\begin{cor}
$SOF^a_N$ is reducible and equi-dimensional.
The number of irreducible components is equal to $2^n$.
\end{cor}

Let us count the number of coordinate flags in $SOF^a_N$. 

\begin{lem}
The number of coordinate flags in $SOF^a_N$ is equal to the number of symmetric Dellac configurations of size
$N\times 2N$. 
\end{lem}
\begin{proof}
We are interested in collections $(I_1,\dots,I_n)$ of subsets of $\{1,\dots,N\}$ such that 
$I_k\subset I_{k+1}\cup\{k+1\}$ and $a\in I_n$ implies $2n+1-a\notin I_n$. As shown above
the number of such collections is equal to the number of symmetric Dellac configurations 
(these are exactly the numbers $r_n$ and $l_n$ from \cite{RZ,BF}).  
\end{proof}

Finally, let us give a combinatorial rule to compute the Poincar\'e polynomials of the varieties $SOF^a_N$.
We construct cellular decomposition of the varieties $SOF^a_N$ in the same way as we did in other types:
if a coordinate flag is given by a collection $\bI=(I_k)_{k=1}^n$ ($N=2n$ or 
$N=2n+1$), then the cell $C(\bI)$ is the subset of  $SOF^a_N$ consisting of subspaces $(V_k)_{k=1}^n$ such that $V_k$
has a basis of the form
\[
e_a+\sum_{b<_k a} c_{a,b} e_b,\ a\in I_k.
\] 
One easily sees that each $C(\bI)$ is isomorphic to an affine cell whose dimension can be computed 
via the following statistics $\widetilde{inv}'$ on the set $SDC_N$.

\begin{dfn}
Let $D\in SDC_N$. Then 
$\widetilde{inv}'(D)$ is equal to the half of the number of inversions which are NOT symmetric with respect to the center.  
\end{dfn} 

For example, the values of the statistics $\widetilde{inv}'$ on the configurations of Figure \ref{fig:SpDC4}
are given by 
\begin{gather*}
0 \ 0 \ 1 \ 1 \ 1\\
1\ 2\ 2\ 2\ 2.
\end{gather*}

Recall that to each $\bI\in {\mathcal I}_N$ we have attached an element $S(\bI)\in SDC_N$.

\begin{thm}
Each $C(\bI)$ is an affine cell. 
$SOF^a_N$ is equal to the  disjoint union $C(\bI)$ for all $\bI\in {\mathcal I}_N$.
The dimension of $C(\bI)$ is equal to $\widetilde{inv}'(S(\bI))$.
\end{thm}
\begin{proof}
The first and the second claims are proved in the same way as Lemma \ref{lem:ac}.
The proof of the third claim is analogous to the proof of Theorem \ref{thm:dimDellac}.
\end{proof}

We thus arrive at the following formula for the Poincar\'e polynomials.

\begin{cor}
$P(SOF^a_N)=\sum_{D\in SDC_N} q^{\widetilde{inv}'(D)}$.
\end{cor}

\appendix
\section{Examples of Poincar\'e polynomials}
Based on the computations below, we put forward the following conjecture (which we are not
able to prove even in type A).

\begin{conj}
The Poincar\'e polynomials in all five types (type $A$, odd and even symplectic and orthogonal)
are unimodular.
\end{conj}

\subsection{Type A}
\begin{align*}
P_{F^a_1}(q) =& 1,\\
P_{F^a_2}(q) =& q+1,\\
P_{F^a_3}(q) =& q^3+3q^2+2q+1,\\
P_{F^a_4}(q) =& q^{6}+6q^5+10q^4+10q^3+7q^2+3q+1.
\end{align*}

\subsection{Odd symplectic}
\begin{align*}
P_{SpF^a_1}(q) =& 1,\\
P_{SpF^a_3}(q) =& q^2 + q + 1,\\
P_{SpF^a_5}(q) =& q^6 + 3q^5 + 5q^4 + 5q^3 + 4q^2 + 2q + 1,\\
P_{SpF^a_7}(q) =& 	
q^{12} + 6q^{11} + 16q^{10} + 29q^{9} + 40q^{8} + 45q^{7} \\&+ 43q^6 + 35q^5 +
25q^4 + 15q^3 + 8q^2 + 3q + 1.
\end{align*}

\subsection{Even symplectic}
\begin{align*}
P_{SpF^a_2}(q) =& q+1,\\
P_{SpF^a_4}(q) =& q^4 + 3q^3 + 3q^2 + 2q + 1,\\
P_{SpF^a_6}(q) =& q^9 + 6q^8 + 13q^7 + 18q^6 + 20q^5 + 17q^4 + 12q^3 + 7q^2 + 3q +
1,\\
P_{SpF^a_8}(q) =& q^{16} + 10q^{15} + 36q^{14} + 79q^{13} + 134q^{12} + 186q^{11} + 220q^{10} +
229q^9 \\&+ 211q^8 + 175q^7 + 130q^6 + 87q^5 + 52q^4 + 27q^3 +
12q^2 + 4q + 1.
\end{align*}

\subsection{Odd orthogonal}
\begin{align*}
P_{SOF^a_1}(q) =& 1,\\
P_{SOF^a_3}(q) =& 2q+1,\\
P_{SOF^a_5}(q) =& 4q^4 + 7q^3 + 6q^2 + 3q + 1,\\
P_{SOF^a_7}(q) =& 8q^9 + 27q^8 + 47q^7 + 56q^6 + 52q^5 + 38q^4 + 23q^3 + 11q^2 +
4q + 1.
\end{align*}

\subsection{Even orthogonal}
\begin{align*}
P_{SOF^a_2}(q) =& 2,\\
P_{SOF^a_4}(q) =& 4q^2 + 4q + 2,\\
P_{SOF^a_6}(q) =& 8q^6 + 20q^5 + 26q^4 + 22q^3 + 14q^2 + 6q + 2,\\
P_{SOF^a_8}(q) =& 16q^{12} + 68q^{11} + 150q^{10} + 230q^9 + 276q^8 + 272q^7 \\&+ 228q^6 + 164q^5 + 102q^4 + 54q^3 + 24q^2 + 8q + 2.
\end{align*}

\section*{Acknowledgments}
The work of EF was partially supported by the grant RSF-DFG 16-41-01013.
The work on the first two sections was supported by the Russian Academic Excellence Project '5-100'.


\begin{thebibliography}{99}
\bibitem[Ba]{Ba}
D. Barsky, {\it Congruences pour les nombres de Genocchi de 2e esp\`ece},
Groupe d'\'etude d'Analyse ultram\'etrique,
8e ann\'ee, no. 34, 1980/81, 13 pp.

\bibitem[B1]{B1}
A.~Bigeni, {\em Enumerating the symplectic Dellac configurations}, arXiv:1705.03804.

\bibitem[B2]{B2}
A.~Bigeni, {\em A generalization of the Kreweras triangle through the universal $sl_2$ weight system},
arXiv:1712.05475.

\bibitem[B3]{B3}
A.~Bigeni, {\em Combinatorial interpretations of the Kreweras triangle in terms of subset tuples},
arXiv:1712.01929.

\bibitem[B4]{B4}
A.~Bigeni, {\em Combinatorial study of the Dellac configurations and the q-extended normalized median Genocchi numbers},
Electronic Journal of Combinatorics, 2014, Vol. 21, No. 2. 
    
\bibitem[BF]{BF}
A.~Bigeni, E.Feigin, {\it Symmetric Dellac configurations}, in preparation.		
		
\bibitem[CFR1]{CFR1}
G.~Cerulli Irelli, E.Feigin, M. Reineke, {\em Quiver Grassmannians and degenerate
flag varieties}, Algebra \& Number Theory 6 (2012), no. 1, 165--194.

\bibitem[CFR2]{CFR2}
G.~Cerulli Irelli, E.Feigin, M. Reineke, {\em Isotropic quiver Grassmannians and degenerate flag varieties of type C},
in preparation.

\bibitem[CL]{CL}
G.~Cerulli Irelli, M.~Lanini,
{\em Degenerate flag varieties of type A and C are Schubert varieties},
Int. Math. Res. Not. (2015), 6353--6374.

\bibitem[CLL]{CLL}
G.~Cerulli Irelli, M.~Lanini, P.~Littelmann,
{\em Degenerate flag varieties and Schubert varieties: a characteristic free approach},
Pacific J. Math. 284 (2016) 283--308.

\bibitem[De]{De} H.~Dellac, Problem 1735, L'Interm\'ediaire des Math\'ematiciens, 7 (1900), 9--10.

\bibitem[Du]{Du} D.~Dumont, {\em Interpr\'etations combinatoires des nombres de Genocchi}, Duke
Math. J. 41 (1974), 305--318.

\bibitem[DR]{DR}
D. Dumont and A. Randrianarivony,
{\it D\'erangements et nombres de Genocchi}, Discrete Math. 132 (1994), 37--49.

\bibitem[DZ]{DZ}
D. Dumont, J. Zeng, {\it Further results on Euler and Genocchi numbers},
Aequationes Mathemicae 47 (1994), 31--42.

\bibitem[FF]{FF}  X.~Fang, G.~Fourier, 
{\em Torus fixed points in Schubert varieties and normalized
median Genocchi numbers}, S\'em. Lothar. Combin. 75 (2015), Art. B75f, 12 pp.

\bibitem[FFi]{FFi}
E.~Feigin, M.~Finkelberg, {\em Degenerate flag varieties of type A: Frobenius splitting and BWB theorem},
Mathematische Zeitschrift, vol. 275 (2013), no. 1--2, pp. 55--77.

\bibitem[F1]{F1} E.~Feigin, {\em Degenerate flag varieties and the median Genocchi numbers}, Math.
Res. Lett. 18 (2011), no. 6, 1163--1178.

\bibitem[F2]{F2} E.~Feigin, {\em The median Genocchi numbers, q-analogues and continued fractions},
European J. Combin., 33(8), 1913--1918, 2012.

\bibitem[F3]{F3}
E.~Feigin, {\em ${\mathbb G}_a^M$ degeneration of flag varieties},  Selecta Mathematica, 
New Series, vol. 18 (2012), no. 3, pp. 513--537.


\bibitem[FFL]{FFL} E.~Feigin, M.~Finkelberg, P.~Littelmann, {\em Symplectic degenerate flag varieties},
Canad. J. Math. 66 (2014), no. 6, 1250--1286.

\bibitem[H]{H}
C. Hague, {\em Degenerate coordinate rings of flag varieties and Frobenius splitting}, Selecta Math.
(N.S.) 20 (2014), no. 3, 823--838.

\bibitem[HZ1]{HZ1}
G.-N. Han and J. Zeng, {\it On a $q$-sequence that generalizes the median Genocchi numbers},
Ann. Sci. Math. Qu\'ebec 23 (1999), 63--72.

\bibitem[HZ2]{HZ2}
G.~Han, J.~Zeng, {\it q-Polynomes de Gandhi et statistique de Denert}, Discrete Math.
205 (1999), no. 1--3, 119-143.

\bibitem[K]{K}
G. Kreweras, {\it Sur les permutations compt\'ees par les nombres de Genocchi de 1-i\`ere et
2-i\`eme esp\`ece}, Europ. J. Combinatorics 18 (1997), 49--58.


\bibitem[LS]{LS}
M.~Lanini, E.~Strickland, {\em  Cohomology of the flag variety under PBW degenerations}, 
arXiv:1706.07079.
    
\bibitem[M]{M}		
I.~Mihai, {\em Odd symplectic flag manifolds}, Transform. Groups 12 (2007), no. 3, 573--599.
		
\bibitem[OEIS1]{OEIS1} The On-Line Encyclopedia of Integer Sequences, https://oeis.org/A098279.

\bibitem[OEIS2]{OEIS2} The On-Line Encyclopedia of Integer Sequences, http://oeis.org/A005773.

\bibitem[Pe]{Pe}
C.~Pech, {\em Quantum cohomology of the odd symplectic Grassmannian of lines}, 
J. Algebra 375, 188--215.

\bibitem[Pr]{Pr}
R.~Proctor, {\em Odd symplectic groups}, Invent. Math. 92 (1988), no. 2, 307--332.


\bibitem[RZ]{RZ} A.~Randrianarivony, J.~Zeng, {\em Une famille de polyn\^omes qui interpole plusieurs
suites classiques de nombres}, Adv. in Appl. Math. 17 (1996), no. 1, 1--26.   

\bibitem[ZZ]{ZZ} 
J.~Zeng, J.~Zhou, {\it A q-analog of the Seidel generation of Genocchi numbers}. Eur. J. Comb., 2006, pp. 364~381 

\end{thebibliography}
\end{document}